\documentclass[11pt,a4paper]{amsart}
\usepackage[margin=1in]{geometry}
\usepackage{amsmath,amssymb,amsthm}
\usepackage{booktabs}
\usepackage[colorlinks,linkcolor=blue,citecolor=blue]{hyperref}
\usepackage{graphicx}
\usepackage{float}

\newtheorem{thm}{Theorem}[section]

\numberwithin{thm}{section}

\newtheorem{lemma}[thm]{Lemma}
\newtheorem{remark}[thm]{Remark}
\newtheorem{proposition}[thm]{Proposition}
\newtheorem{conjecture}[thm]{Conjecture}

\numberwithin{equation}{section}
\numberwithin{table}{section}

\DeclareMathOperator{\Area}{Area}
\DeclareMathOperator{\Real}{Re}
\DeclareMathOperator{\Imag}{Im}
\newcommand{\HH}{\mathbb{H}}
\newcommand{\RR}{\mathbb{R}}
\newcommand{\CC}{\mathbb{C}}

\newcommand{\DD}{\ensuremath{\mathrm{D}}}
\newcommand{\Fone}{F_1}
\newcommand{\Ftwo}{F_2}
\newcommand{\Sph}{\mathbb{S}}

\title[Smyth's free--boundary minimal disk in the regular tetrahedron]
{An exact hypergeometric Weierstrass representation\\
for Smyth's free--boundary minimal disk\\
in the regular tetrahedron}
\author{Steven Finch}
\address{MIT Sloan School of Management, Cambridge, MA 02139, USA}
\email{steven\_finch\_math@outlook.com}
\date{\today}

\begin{document}
\maketitle

\begin{abstract}
Smyth (1984) proved that every tetrahedron $T\subset\RR^3$ contains three
embedded minimal disks meeting $\partial T$ orthogonally, obtained as
conjugate surfaces of Plateau solutions in a skew quadrilateral $Q$ whose 
edges are $c_i\nu_i$, where $\varphi_1,\varphi_2,\varphi_3,\varphi_4$ are 
faces of $T$, $c_i=\Area(\varphi_i)$ and $\nu_i$ are outward normal vectors. 
For the \emph{regular} tetrahedron we give the Weierstrass data explicitly. 
The Gauss map is a Schwarz triangle function for a circular--arc triangle 
with angles $(\pi/4,\ \pi/2,\ \pi\beta/2)$, where
$\beta=\arccos(-1/3)/\pi$; the hypergeometric parameters are linear in
$\beta$ and there is \emph{no accessory parameter and no free parameter}.
We record a closed form for $\kappa$ in $\Gamma$-values, a Wronskian
closed form for $g'$, the exact total curvature $2\pi(2\beta-1)$, the
symmetric--square structure of the data, and a rigidity statement: since
$\beta/2$ is irrational, the projective monodromy is dense and the data
\emph{cannot} be algebraic. Numerically the representation extends to
28 digits -- $\Area/\ell^{2}=0.2172341554075574483760009126\ldots$ -- 
confirming a 2013 Surface Evolver computation of Ken~Brakke.
\end{abstract}

\section{Numerical cross-check}

\begin{table}[h]
\centering
\begin{tabular}{lll}
\toprule
Source & $\Area/\ell^{2}$ & method\\
\midrule
K.~Brakke, priv. comm. (2013) & $0.217234155407559\ldots$ & Surface Evolver\\
Present representation & $0.217234155407557448\ldots$ & conjugate Plateau, hypergeom\\
\midrule
relative difference & $7.4\times10^{-15}$ &\\
\bottomrule
\end{tabular}
\end{table}

\noindent
\noindent
The two computations agree to $14$ digits; the absolute difference $1.6\times10^{-15}$
(relative $7.4\times10^{-15}$) is consistent with discretization error in Brakke's 
software package, whose convergence at the vertices of $\partial\Sigma^{*}$ is
only $O(h^{1-\beta})$, $1-\beta=0.3918\ldots$ (Remark~\ref{rem:vertex}).
Two entirely independent methods --- a discrete variational relaxation and
an exact analytic representation --- therefore confirm one another, and the
representation extends the constant from $14$ to $28$ digits.

\section{The representation}

\subsection{Constants}

Put
\begin{equation}
\beta=\frac{\arccos(-\tfrac13)}{\pi}
 =0.608173447969392729829\ldots,
\end{equation}
so that $\pi\beta$ is the angle $\approx 109.4712^\circ$ between adjacent outward face normals, and set
\begin{equation}\label{eq:abc}
a=\frac18-\frac{\beta}{4},\qquad
b=\frac58-\frac{\beta}{4}=a+\frac12,\qquad
c=\frac34 .
\end{equation}
Numerically $a=-0.0270433619923\ldots$ and $b=0.4729566380077\ldots .$ Note
the three exponent differences
\begin{equation}\label{eq:expdiff}
1-c=\frac14,\qquad b-a=\frac12,\qquad c-a-b=\frac{\beta}{2}
\end{equation}
at $t=0,\infty,1$ respectively. Let $_2F_1$ denote the ordinary hypergeometric function; write
\begin{equation}\label{eq:OrdHyp}
\Fone(t)={}_2F_1\!\left(a,\,b;\,\tfrac34;\,t\right),
\qquad
\Ftwo(t)={}_2F_1\!\left(a+\tfrac14,\,b+\tfrac14;\,\tfrac54;\,t\right).
\end{equation}

\subsection{Domain and uniformizing coordinate}

The fundamental domain is the $\tfrac18$-sector of the unit disk
\begin{equation}\label{eq:fundmn}
S=\Bigl\{w\in\CC:\ |w|<1,\ 0<\arg w<\tfrac{\pi}{4}\Bigr\},
\end{equation}
mapped to the upper half-plane by
\begin{equation}\label{eq:tofw}
t=\frac{4w^{4}}{(1+w^{4})^{2}}\ :\ S\ \xrightarrow{\ \cong\ }\ \HH,
\qquad
t^{1/4}=\frac{\sqrt2\,w}{(1+w^{4})^{1/2}} .
\end{equation}
The hypergeometric differential equation has singular points $t=0,1,\infty$; the three
sides of $S$ map to the three intervals of $\partial\mathbb{H}=\mathbb{R}\cup\{\infty\}$
that they separate:
\[
\{\arg w=0\}\mapsto(0,1),\qquad
\{|w|=1\}\mapsto(1,\infty),\qquad
\{\arg w=\tfrac\pi4\}\mapsto(-\infty,0),
\]
the three corners $w=0,1,e^{i\pi/4}$ of $S$ going to $t=0,1,\infty$
respectively.  Geometrically, on $\Sigma$ the two rays are a planar symmetry curve and a
straight line, while $|w|=1$ is an edge of $Q$; conjugation
interchanges these roles, so on $\Sigma^{*}$ the arc $|w|=1$ is the
\emph{free boundary} (\S\ref{subsec:two-srf}).

\subsection{Weierstrass data}\label{subsec:Wrsdat}

\begin{equation}\label{eq:data}
\boxed{\;
g=\kappa\,t^{1/4}\,\frac{\Ftwo(t)}{\Fone(t)},
\qquad
dh=\lambda\,t^{-1/2}\,(1-t)^{-\frac{1+\beta}{2}}\,\Fone(t)\,\Ftwo(t)\,dt
\;}
\end{equation}
with $\kappa$ given in \eqref{eq:kappa} below and $\lambda$ a nonzero complex constant.
Pulled back to $w$ through \eqref{eq:tofw}, and abbreviating
\begin{equation}\label{eq:rwabrv}
R(w)=(1-w^{4})^{-\beta}\,(1+w^{4})^{\beta},
\end{equation}
the three basic $1$-forms are
\begin{align}
\frac{dh}{g}&=4\sqrt2\,\frac{\lambda}{\kappa}\;\Fone^{2}\,
   \frac{R\,dw}{(1+w^{4})^{1/2}},\label{eq:w1}\\[2pt]
dh&=8\lambda\,w\;\Fone\Ftwo\,
   \frac{R\,dw}{1+w^{4}},\label{eq:w2}\\[2pt]
g\,dh&=8\sqrt2\,\lambda\kappa\;w^{2}\,\Ftwo^{2}\,
   \frac{R\,dw}{(1+w^{4})^{3/2}} .\label{eq:w3}
\end{align}
The surface is
\begin{equation}\label{eq:X}
X(w)=\Real\int^{w}\Phi,
\qquad
\Phi=\Bigl(\tfrac12\bigl(\tfrac{dh}{g}-g\,dh\bigr),\
           \tfrac{i}{2}\bigl(\tfrac{dh}{g}+g\,dh\bigr),\
           dh\Bigr).
\end{equation}
Because $\bigl(\tfrac{dh}{g}\bigr)(g\,dh)=dh^{2}$ holds
term-by-term in \eqref{eq:w1}--\eqref{eq:w3}, the conformality relation
$\Phi\cdot\Phi\equiv0$ is an algebraic identity, not a numerical
coincidence.

\subsection{The two surfaces}\label{subsec:two-srf}

The parameter $\lambda$ controls the associate family; only its phase
matters, its modulus being an overall homothety:
\begin{itemize}
\item $\lambda\in\RR$: Smyth's Plateau disk $\Sigma$ spanning the
equilateral skew quadrilateral $Q$ whose four edges are the
tetrahedral directions (side $\sqrt3$, interior angles 
$\arccos\tfrac13 \approx 70.53^\circ$, two equal perpendicular diagonals)
\item $\lambda\in i\RR$: the \textbf{free-boundary minimal disk
$\Sigma^{*}$ in a regular tetrahedron} $\tilde{T}$, meeting all four faces
orthogonally (see Figures 1 and 2).
\end{itemize}
The sign of $\lambda$ fixes an orientation of $\Sigma^{*}$; we take
$\lambda=-i$ throughout \S7--\S8, so that $Y>0$ in \eqref{eq:YJ}, and
all lengths and areas are understood as positive.

\clearpage

\begin{figure}[H]
	\vspace*{3cm}   
	\centering
	\includegraphics[width=\textwidth,height=\textheight,keepaspectratio]{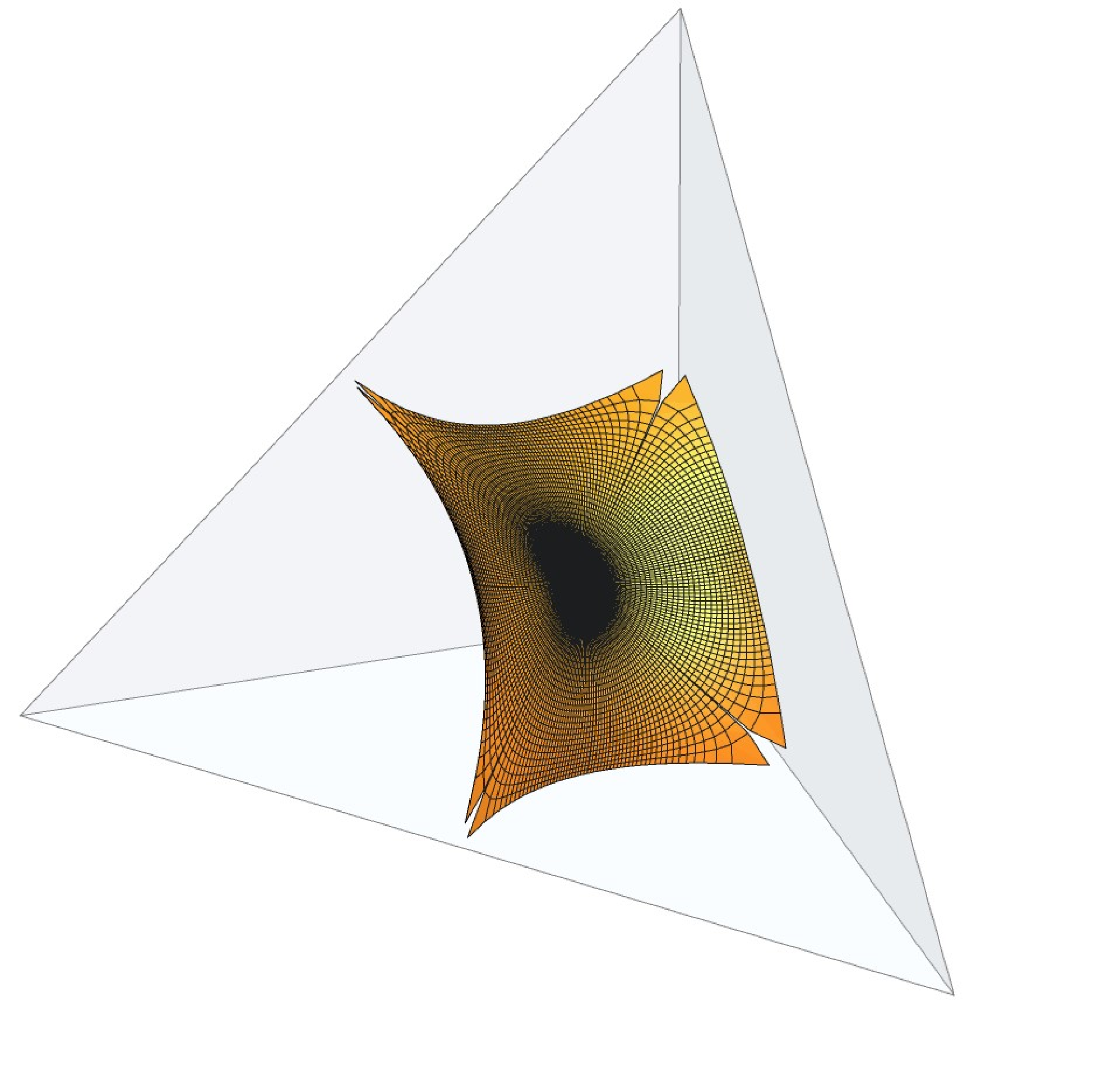}
	  \caption[Minimal disk in a tetrahedron]{One view of Smyth's free--boundary minimal disk 
		$\Sigma^{*}$ in $\tilde{T}$,\\ a regular tetrahedron\footnotemark . Its four corners 
		are midpoints of four of the six\\ edges of $\tilde{T}$; the two remaining edges, 
		an opposite pair, are untouched.}
\end{figure}
\footnotetext{$\Sigma^{*}$ has free boundary not in $T$ itself but 
	in a tetrahedron $\tilde{T}$ homothetic to $T$, whose size is fixed by $|\lambda|$; 
	rescaling $\Sigma^{*}$ by $\ell(T)/\ell(\tilde{T})$ places it in $T$. Since we report 
	only the scale-invariant ratio $\Area/\ell^{2}$, the distinction is immaterial
	numerically, but we keep the tilde as a reminder that the ambient
	tetrahedron is an output of the construction, not an input.}

\clearpage 

\subsection{The single constant, in $\Gamma$-values}

Gauss's theorem at $t=1$ evaluates both series; the two factors
$\Gamma(\beta/2)$ cancel, leaving
\begin{equation}\label{eq:kappa}
\kappa=(\sqrt2-1)\,\frac{\Fone(1)}{\Ftwo(1)}
=(\sqrt2-1)\,
\frac{\Gamma\!\left(\frac34\right)
      \Gamma\!\left(\frac78+\frac{\beta}{4}\right)
      \Gamma\!\left(\frac38+\frac{\beta}{4}\right)}
     {\Gamma\!\left(\frac54\right)
      \Gamma\!\left(\frac58+\frac{\beta}{4}\right)
      \Gamma\!\left(\frac18+\frac{\beta}{4}\right)}
=0.23973032106112410901\ldots .
\end{equation}
The normalization is fixed by $g=\sqrt2-1$ at $w=1$, i.e.,\ the corner of
the spherical square.

\subsection{Closed form for $g'$}

Since $y_{1}=t^{1/4}\Ftwo$ and $y_{2}=\Fone$ are the two Frobenius
solutions at $t=0$ of one and the same hypergeometric equation
(indices $0$ and $1-c=\tfrac14$), their Wronskian is
$W=(1-c)\,t^{-c}(1-t)^{c-a-b-1}$, and pulling back through
\eqref{eq:tofw} gives
\begin{equation}\label{eq:gprime}
\boxed{\;
g'(w)=\sqrt2\,\kappa\,
\frac{(1-w^{4})^{\beta-1}\,(1+w^{4})^{\frac12-\beta}}{\Fone(t)^{2}}
\;}
\end{equation}
with the check $g'(0)=\sqrt2\,\kappa$. No differentiation of
hypergeometric functions is ever required.

\section{Why there is no free parameter}

By \eqref{eq:expdiff}, $g$ is precisely the \emph{Schwarz triangle
function} of the circular--arc triangle with interior angles
\begin{equation}\label{eq:trifct}
\left(\frac{\pi}{4},\ \frac{\pi}{2},\ \frac{\pi\beta}{2}\right)
=\bigl(45^{\circ},\ 90^{\circ},\ \approx 54.7356^{\circ}\bigr),
\end{equation}
whose image is one eighth of a spherical square of the spherical
cuboctahedron (with corners 
$\pm(\sqrt2-1)$, $\pm i(\sqrt2-1)$ in
stereographic coordinates, all angles $\arccos(-\tfrac13)$).

A \emph{three}-cornered circular polygon has a rigid uniformizing map:
the Schwarzian equation
\[
\{f,w\}=\sum_{j}\Bigl[\frac{1-\theta_{j}^{2}}{2(w-w_{j})^{2}}
   +\frac{\beta_{j}}{w-w_{j}}\Bigr]
\]
has no free accessory parameter when $n=3$ \cite[Ch.~V]{Nehari}, whereas 
a four-cornered one does, and its determination requires isomonodromic
(Painlev\'e~VI) methods \cite{AnselmoNelsonCrowdy}. The $D_{2d}$
symmetry of $Q$ (order $8$) is exactly what reduces the four
corners of the Gauss image to three. \textbf{This is the structural
reason the surface admits a closed form at all.}

\begin{proposition}[Exact total curvature]\label{prop:extocrv}
\begin{equation}
\left|\int_{\Sigma^{*}}\!K\,dA\right|
=8\times\text{(spherical excess)}
=8\pi\!\left(\tfrac{\beta}{2}-\tfrac14\right)
=2\pi(2\beta-1)
=4\arccos\!\left(-\tfrac13\right)-2\pi
=1.3593475\ldots
\end{equation}
Equivalently $\deg g=1$ onto the spherical cuboctahedral square; the
Gauss--Bonnet balance is
$\int K=2\pi-4\bigl(\pi-\arccos\tfrac13\bigr)$ for a disk with four
free-boundary arcs.
\end{proposition}

\section{Symmetric--square structure}

Let
\[
y_{1}=\sqrt{\kappa}\;t^{1/4}\Ftwo(t),
\qquad
y_{2}=\frac{1}{\sqrt{\kappa}}\,\Fone(t)
\]
be the Frobenius basis at $t=0$ of \S\ref{subsec:Wrsdat}, rescaled so that
$g=y_{1}/y_{2}$ exactly, and put
\begin{equation}
dq=\lambda\,t^{-3/4}(1-t)^{-\frac{1+\beta}{2}}\,dt .
\end{equation}
Then \eqref{eq:data} reads
\begin{equation}
\frac{dh}{g}=y_{2}^{2}\,dq,\qquad
dh=y_{1}y_{2}\,dq,\qquad
g\,dh=y_{1}^{2}\,dq,\qquad
g=\frac{y_{1}}{y_{2}}
\end{equation}
so that
\begin{equation}\label{eq:symsq}
\Phi=\Bigl(\tfrac12(y_{2}^{2}-y_{1}^{2}),\
           \tfrac{i}{2}(y_{2}^{2}+y_{1}^{2}),\
           y_{1}y_{2}\Bigr)\,dq .
\end{equation}
Thus the entire surface is the \emph{symmetric square} of a single
second-order Fuchsian equation, twisted by one algebraic $1$-form;
$\Phi\cdot\Phi=0$ is then the identity
$(y_{2}^{2}-y_{1}^{2})^{2}-(y_{2}^{2}+y_{1}^{2})^{2}+4y_{1}^{2}y_{2}^{2}=0$.
This is the natural object on which to run a monodromy or isomonodromy
analysis.

\section{A rigidity theorem: no algebraic data exist}

\begin{thm}[Non-algebraicity]\label{thm:nonalg} 
The Gauss map of Smyth's minimal disk in the regular tetrahedron is not
an algebraic function. In particular there is \emph{no} representation
with $g$ rational and $dh=dw/\sqrt{P(w)}$ for any polynomial $P$.
\end{thm}

\begin{proof}[Reason]
By Schwarz's classification a triangle function is algebraic if and only
if its projective monodromy group is finite. Here the local exponent
difference at $t=1$ is
$\beta/2=\arccos(-\tfrac13)/(2\pi)$, which is irrational; hence the local
monodromy about $t=1$ has infinite order and the projective monodromy is
a dense, non-discrete subgroup of $PSU(2)$.
\end{proof}

\begin{remark}
Equivalently: the dihedral angle $\arccos\tfrac13$ of the regular
tetrahedron is not a rational multiple of $\pi$, so Schwarz reflection in
the four faces generates a non-discrete group and the extended surface is
\emph{not} triply periodic. This rules out at the outset any
$tD/rPD$-type ansatz such as $g=w$, $dh=dw/\sqrt{w^{8}+2\nu w^{4}+1}$
for a real parameter $\nu$
\cite{Karcher,FogdenHyde,ChenWeber}, and explains structurally why no 
formula has ever appeared in the classical literature.
\end{remark}

\section{A one-parameter family; the flat square as a limit}

Nothing in \S2 used the particular value of $\beta$. Letting $\beta$
vary gives Smyth disks in \emph{tetragonal disphenoids} (tetrahedra with
four congruent isosceles faces), for which $c_{1}=c_{2}=c_{3}=c_{4}$
automatically, so Smyth's construction applies verbatim.

\begin{table}[h]
\centering
\begin{tabular}{llp{6.6cm}}
\toprule
$\beta$ & triangle angles & character\\
\midrule
$\to\tfrac12$ & $(\tfrac\pi4,\tfrac\pi2,\tfrac\pi4)$, Euclidean $(2,4,4)$ &
$\int K\to0$: degenerates to the \emph{flat} mid-edge square,
$\Area/\ell^{2}=\tfrac14$\\
$\tfrac23$ & $(\tfrac\pi4,\tfrac\pi2,\tfrac\pi3)$, octahedral $(2,3,4)$ &
finite monodromy $\Rightarrow$ $g$ algebraic; classical Schwarz-type
(\emph{conjectural} identification, unverified)\\
$\dfrac{\arccos(-1/3)}{\pi}$ & $(\tfrac\pi4,\tfrac\pi2,\ \approx 54.7356^{\circ})$ &
\textbf{regular tetrahedron}: dense monodromy, transcendental\\
\bottomrule
\end{tabular}
\end{table}

For a tetragonal disphenoid the six edges fall into two $D_{2d}$ orbits:
four congruent edges, along which the dihedral angle is $\delta_{1}$, and
two opposite mutually perpendicular edges, along which it is $\delta_{2}$.
Writing the vertices as $(\pm p,0,q)$, $(0,\pm p,-q)$ and $\mu=2q/p$,
the outward face normals are $(0,\pm\mu,1)$, $(\pm\mu,0,-1)$
divided by $\sqrt{1+\mu^{2}}$, whence
\[
\cos\pi\beta=-\frac{1}{1+\mu^{2}},\qquad
\delta_{1}=\pi(1-\beta),\qquad
\cos\delta_{2}=1+2\cos\pi\beta=\frac{\mu^{2}-1}{\mu^{2}+1}.
\]
Thus $\pi\beta$ is the angle between adjacent outward face normals, and
\[
\beta=1-\frac{\delta_{1}}{\pi}
\]
is the single shape parameter of the family. It decreases from
$\beta=\tfrac23$ at $\mu=1$ ($\delta_{1}=\tfrac\pi3$,
$\delta_{2}=\tfrac\pi2$; edges $4$ and $2\sqrt3$) through the regular
tetrahedron $\mu=\sqrt2$, $\delta_{1}=\delta_{2}=\arccos\tfrac13$, to
$\beta\to\tfrac12$ as $\mu\to\infty$, where the disphenoid elongates,
the four faces become perpendicular to the mid-edge plane, and
$\int K\to0$. (Letting $\mu\to0$ would continue the family to
$\beta\to1$; we do not consider $\beta>\tfrac23$.) 

In the following, we speak of a surface-valued map being ``real-analytic'' 
in the sense that the Weierstrass data
\eqref{eq:data} --- equivalently the map
$(\beta,w)\mapsto\Phi_{\beta}(w)$ on $(\tfrac12,\tfrac23]\times S$ ---
depends in a jointly real-analytic manner on $(\beta,w)$, and hence so do
$X^{*}_{\beta}$ and all derived quantities.

Smyth's three disks arise from the three ways of closing the four vectors
$c_{i}\nu_{i}$ into a skew quadrilateral, indexed by which pair of normals
is diagonal rather than adjacent. The symmetry group of $T$ permutes
these three: for the regular tetrahedron it acts transitively and the
three disks are congruent, whereas for a general disphenoid $D_{2d}$ fixes
one ordering and interchanges the other two, leaving two congruence
classes. Only the distinguished ordering has all four angles of its
spherical quadrilateral equal to $\pi\beta$; the other two have angles
alternating between $\pi\beta$ and $\pi-\delta_{2}$, and stabilizer of
order $4$. In the regular case
$\cos\pi\beta=(1-\mu^{2})/(1+\mu^{2})=-\tfrac13$ and all three coincide.

\begin{conjecture}\label{conj:family}
	For each $\beta$ let $\Sigma^{*}_{\beta}$ denote the distinguished
	Smyth disk, the one invariant under the full symmetry group $D_{2d}$ of
	the disphenoid.\footnote{The other two disks are almost certainly harder, 
	not merely different: their Gauss quadrilateral has unequal angles and only 
	order-4 symmetry, so the fundamental domain is a circular-arc \emph{quadrilateral}, 
	not a triangle -- and then the accessory parameter of §3 returns.}
	The map $[\tfrac12,\tfrac23]\ni\beta\mapsto\Sigma^{*}_{\beta}$ 
	is a real-analytic $1$-parameter deformation of free-boundary minimal
	disks in tetragonal disphenoids, with natural parameter the total
	curvature $\int K=-2\pi(2\beta-1)$ running from $0$ to $-2\pi/3$. Its two
	endpoints are the flat mid-edge square ($\beta=\tfrac12$) and the
	fundamental piece of Schwarz's P--surface ($\beta=\tfrac23$); the
	regular-tetrahedron disk is the interior member
	$\beta=\arccos(-\tfrac13)/\pi$. Equivalently, in the conjugate picture,
	$\Sigma_{2/3}$ is the fundamental piece of Schwarz's
	D--surface.\footnote{This is not to be confused with a separate occurrence 
	of D: cf.\ \cite{GKD}, whose Fig.~1a (p. 543) shows 
	the fundamental domain of D, characterized later  (p. 545) as 
	``inscribed in a tetragonal disphenoid, an isohedral tetrahedron 
	with isosceles triangle faces''.  We have not studied this occurrence; 
	its $\beta$ value is not yet known.}
\end{conjecture}

\section{Scale-free expressions for the area}

Along $\partial\Sigma^{*}$ the unit normal $N$ lies \emph{in} the face,
hence the conormal is parallel to the face normal $n$; with
$X\cdot n=-r$ constant on the face this gives, from
$\Area=\tfrac12\oint X\cdot\partial_{r}X\,d\theta$,
\begin{equation}\label{eq:arearL}
\Area(\Sigma^{*})=\tfrac12\,r\,L,
\qquad
\frac{\Area}{\ell^{2}}=\frac{L}{48\,r}
\end{equation}
where $L=|\partial\Sigma^{*}|$, $r$ is the inradius of $\tilde{T}$ and
$\ell=2\sqrt6\,r$ its edge.

Identity \eqref{eq:arearL} is load-bearing in what follows: it is what
permits the area to be obtained with \emph{no} surface integral and,
after Proposition~\ref{prop:twoint}, with no integral along the free
boundary either. Its two inputs --- planarity of each boundary arc, and
orthogonality along it --- are independent of \eqref{eq:arearL} and are
verified numerically in \S8 to $11$ and $12$ digits respectively; the
resulting area agrees with a direct Dirichlet-energy quadrature to $12$
digits.

\begin{proposition}[Corners]\label{prop:corners}
	The four corners of $\partial\Sigma^{*}$ are the midpoints of four of the
	six edges of $\tilde{T}$. Indeed $\sum_i n_{i}=0$ forces
	$\sum_i X\cdot n_{i}=0$; two offsets equal $-r$ (the corner lies on an
	edge), the $D_{2d}$ mirror fixing that edge makes the other two equal,
	so the offsets are $(-r,-r,r,r)$, whence
	$|X_{\mathrm{corner}}|=r\sqrt3=\ell/(2\sqrt2)$.
\end{proposition}

\noindent
Consequently $\Sigma^{*}$ and the flat mid-edge square share the same
four corner points: the comparison in \S8 is exactly apples‑to‑apples.

\medskip
All of $r$, $\ell$, $L$ and $\Area$ are now expressible by two real
integrals along the two \emph{rays} of $S$, with the free-boundary arc
never entered. Let
\begin{equation}\label{eq:psidef}
\Psi(w)\,dw=\frac12\Bigl(\frac{dh}{g}+g\,dh\Bigr)\Big|_{\lambda=1},
\qquad\text{so that}\qquad \Phi_{2}=i\,\Psi\,dw ;
\end{equation}
explicitly $\Psi=4\sqrt2\,\Fone^{2}R/\bigl(\kappa\sqrt{1+w^{4}}\bigr)
+8\sqrt2\,\kappa\,w^{2}\Ftwo^{2}R/(1+w^{4})^{3/2}$, halved.

\begin{proposition}[Two-integral reduction]\label{prop:twoint}
	Put
	\begin{equation}\label{eq:YJ}
	Y=\int_{0}^{1}\Psi(\rho)\,d\rho,
	\qquad
	J=\int_{0}^{1}\Real\Bigl[e^{3i\pi/4}\,\Psi\bigl(\rho e^{i\pi/4}\bigr)\Bigr]d\rho .
	\end{equation}
	Both integrands are real-valued and real-analytic on $(0,1)$ with
	integrable endpoint behavior, and
	\begin{equation}\label{eq:twoint}
	\boxed{\;
		r=\frac{Y}{\sqrt3},\qquad
		\ell=2\sqrt2\,Y,\qquad
		L=8\sqrt3\,|J|,\qquad
		\Area(\Sigma^{*})=4\,Y\,|J|,\qquad
		\frac{\Area}{\ell^{2}}=\frac{|J|}{2\,Y}\,.}
	\end{equation}
\end{proposition}

\begin{proof}
	On the ray $\arg w=0$ one has $t=4\rho^{4}/(1+\rho^{4})^{2}\in[0,1)$
	real, so $\Fone,\Ftwo,R$ are real; hence $\Phi_{1},\Phi_{3}$ are real and
	$\Phi_{2}=i\Psi\,d\rho$ is purely imaginary. With $\lambda=-i$ (\S2.4)
	this gives $X^{*}(1)=(0,\pm Y,0)$, so by
	Proposition~\ref{prop:corners} $Y=|X_{\mathrm{corner}}|=r\sqrt3$ and
	$\ell=2\sqrt6\,r=2\sqrt2\,Y$.
	
	On the ray $\arg w=\tfrac\pi4$ one has $w^{4}=-\rho^{4}$, so
	$t=-4\rho^{4}/(1-\rho^{4})^{2}<0$; with $dw=e^{i\pi/4}d\rho$ the second
	component of $\Phi$ is $e^{3i\pi/4}\Psi\,d\rho$, whence $J$ is the second
	component of $X_{\Sigma}(e^{i\pi/4})$, since $X_{\Sigma}(0)=0$. By \S2.4
	the arc $|w|=1$ maps to a straight segment of $Q$ in the direction
	$(1,1,1)/\sqrt3$, so $X_{\Sigma}(e^{i\pi/4})-X_{\Sigma}(1)$ has three
	equal components, each $=J$; its length is $\sqrt3\,|J|$ and, the
	$D_{2d}$ orbit of the fundamental domain having order $8$,
	$L=8\sqrt3\,|J|$. Finally \eqref{eq:arearL} gives
	$\Area=\tfrac12(Y/\sqrt3)\,(8\sqrt3\,|J|)=4Y|J|$ and
	$\Area/\ell^{2}=4Y|J|/(8Y^{2})=|J|/(2Y)$.
\end{proof}

\begin{lemma}[Straightness certificate]\label{lem:cert}
	Let $I_{k}$ and $J_{k}$, $k=1,2,3$, denote the three components of the
	ray integrals of \eqref{eq:YJ}, so that $X_{\Sigma}(1)=(I_{1},0,I_{3})$,
	$I_{2}=0$ identically, $J_{2}=J$, and
	$X_{\Sigma}(e^{i\pi/4})=(J_{1},J_{2},J_{3})$. Then
	\begin{equation}\label{eq:cert}
	J_{1}-I_{1}\;=\;J_{2}\;=\;J_{3}-I_{3}.
	\end{equation}
	Numerically all three equal
	$-7.9005445743870256267164139949\ldots$, agreeing to $29$ significant
	digits. Since \eqref{eq:cert} holds if and only if the image of $|w|=1$
	is a straight segment in the direction $(1,1,1)/\sqrt3$, it
	simultaneously certifies $\kappa$, the exponent
	$(1-t)^{-(1+\beta)/2}$ and $R=(1-w^{4})^{-\beta}(1+w^{4})^{\beta}$, the
	pullback \eqref{eq:tofw}, and the two hypergeometric representations used
	for $t<0$ and near $\rho=1$. 
\end{lemma}

\noindent
Equivalently, $I_{1}=2|J|$, $I_{3}=|J|$, $I_{2}=0$, $J_{1}=-J_{2}=|J|$
and $J_{3}=0$, so that $X_{\Sigma}(0)=0$, $X_{\Sigma}(1)=|J|(2,0,1)$ and
$X_{\Sigma}(e^{i\pi/4})=|J|(1,-1,0)$ are precisely the centroid, a vertex
and an edge midpoint of the equilateral skew quadrilateral $Q$, with
$|X_{\Sigma}(1)|/|X_{\Sigma}(e^{i\pi/4})|=\sqrt{5/2}$ and edge length
$c_{i}=2\sqrt3\,|J|$; the last two relations are the ones that link the
two rays, the others holding pointwise by the reality of
$\Fone,\Ftwo,R$.

\begin{remark}[vertex constant and truncation law]\label{rem:vertex}
	The vertex $w=1$ of $Q$ is the only singular endpoint of
	\eqref{eq:YJ}; there $\Psi\sim(1-\rho)^{-\beta}$, and the substitution
	$1-\rho=\tau^{\,p}$, $p=1/(1-\beta)$, removes it \emph{exactly}: the
	factors $\tau^{\,p-1}$ and $(1-\rho^{4})^{-\beta}$ cancel identically,
	and the transformed integrand has the finite limit
	\begin{equation}\label{eq:L0}
	L_{0}=\frac{4\sqrt2}{1-\beta}\;2^{-\beta}\,\Fone(1)\Ftwo(1)
	=14.239368600546262861\ldots,
	\end{equation}
	verified against \eqref{eq:YJ} to $20$ digits. The exponent
	$1-\beta=\arccos(\tfrac13)/\pi$ is exactly the Plateau vertex exponent
	for two straight edges meeting at angle $\pi(1-\beta)$. 
\end{remark}

\noindent
Consequently any scheme truncating at parameter distance $u_{\min}$ from
the vertex incurs a one-signed relative error $\approx 8L_{0}u_{\min}/L$,
while any scheme resolving the vertex at finite mesh scale $h$ incurs
$O(h^{1-\beta})$ with the poor exponent $1-\beta=0.3918\ldots$; this is
why vertex-limited methods stall near six digits. As a check of the law,
an earlier computation of the present data by ODE integration along the
free boundary, truncated at $u_{\min}=10^{-6}$, returned $L$ and
$\Area(\Sigma^{*})$ low by $1.0406\times10^{-6}$ in relative terms,
matching the prediction to five digits. By contrast the integrands of
\eqref{eq:YJ} are analytic at $\rho=0$ and at $\rho=1$ on the ray
$\arg w=\pi/4$, and algebraically desingularized at the vertex on the ray
$\arg w=0$, so \eqref{eq:YJ} converges geometrically and the endpoints
are evaluated, not approached.

\pagebreak

\section{Numerical summary}\label{sec:NumSum}

\begin{table}[h]
\centering
\begin{tabular}{lll}
\toprule
quantity & value & status\\
\midrule
$\beta=\arccos(-1/3)/\pi$ & $0.608173447969392729829\ldots$ & exact input\\
$a$, $b$, $c$ & $-0.0270433620$, $0.4729566380$, $3/4$ & linear in $\beta$\\
$\kappa$ & $0.239730321061124109007\ldots$ & exact, \eqref{eq:kappa}\\
triangle angles & $45^{\circ},\ 90^{\circ},\ \approx 54.7356^{\circ}$ & no accessory parameter\\
$\int K$ & $-1.3593476378164877\ldots$ & $=-2\pi(2\beta-1)$\\
inradius $r$ of $\tilde{T}$ & $10.4987669370161548324712311$ & \\
edge $\ell=2\sqrt6\,r$ & $51.4332438481844713471416932$ & identity exact\\
$L=|\partial\Sigma^{*}|=4c_{i}$ & $109.4731568824076774556327$ & \\
$\Area(\Sigma^{*})$ & $574.666579983902\ldots$ & three routes agree, $12$ digits\\
$\Area/\ell^{2}$ & $0.2172341554075574483760009$ & \\
saving vs.\ flat square & $13.106\%$ & $=1-4\,\Area/\ell^2$\\
\bottomrule
\end{tabular}
\end{table}

\noindent
\textbf{Verified identities.} The straightness certificate
\eqref{eq:cert} returns three values agreeing in all $30$ computed digits
at $\mathrm{prec}=80$ (they split in the $29$th at $\mathrm{prec}=60$),
equivalently $I_{1}=2|J|$ and $I_{3}=|J|$ to that accuracy: the images of
$w=0,1,e^{i\pi/4}$ are exactly the centroid, a vertex and an edge
midpoint of the equilateral skew quadrilateral $Q$, with
$|X_{\Sigma}(1)|/|X_{\Sigma}(e^{i\pi/4})|=\sqrt{5/2}$. The vertex
constant \eqref{eq:L0} agrees with its closed form to $20$ digits. The
relations $r=Y/\sqrt3$, $\ell=2\sqrt2\,Y$, $J_{1}=-J_{2}$, $J_{3}=0$ and
$I_{2}=0$ hold identically by \eqref{eq:twoint} and the reality of
$\Fone,\Ftwo,R$ on the two rays. Independently, the earlier
ODE-and-quadrature computation gave $r$ to $12$ digits, planarity of
each boundary arc to $1.1\times10^{-11}$ relative, orthogonality
$\max|N\cdot n|=2.2\times10^{-12}$, four face normals with pairwise inner
products $-\tfrac13$ (so $\tilde{T}$ is regular), $\Area(\Sigma)=\Area(\Sigma^{*})$
by isometry to $8\times10^{-8}$, and $\Area=\tfrac12 rL$ to $12$ digits;
its $1.04\times10^{-6}$ deficit in $L$ and $\Area$ is accounted for
exactly by Remark~\ref{rem:vertex}. The total curvature
$\int K=2\pi-4\arccos(-\tfrac13)$ is exact by Proposition~\ref{prop:extocrv}.

\section{Ledger of what is and is not established}

\begin{enumerate}
\item The representation \eqref{eq:data}--\eqref{eq:X} with
\eqref{eq:kappa} and \eqref{eq:gprime}: established, and verified
numerically on five independent structural identities and against
an independent Surface Evolver computation.
\item The non-algebraicity theorem (\S5): established.
\item The $\beta$-family and the identification of $\beta=\tfrac23$ with
a classical Schwarz-type surface \cite{Schwarz}: conjectured.
\item Everything concerns \emph{stationarity}. Smyth guaranteed three
free-boundary disks in any tetrahedron  \cite{Smyth1984, Choe2025}; in 
the regular case symmetry makes them congruent, but \emph{global minimality} among
volume-halving competitors (e.g.,\ vertex-cutting caps) is untouched here.
The $S_{4}$ symmetry of $\Sigma^{*}$ does exchange the two complementary
solids, so the volume constraint is inactive and the Lagrange multiplier
vanishes: this is why $H=0$ rather than $H=\text{const}$.
\item No closed form for $0.2172341554\ldots$ is claimed or expected; the
exact content resides in the representation (not the constant).
\end{enumerate}

\pagebreak

\appendix

\section{How the representation was found}
	
	This appendix is expository. It answers four questions --- what a Gauss
	map is, what a Schwarz triangle function is, how the hypergeometric
	function entered, and precisely how the parameters $a,b,c$ of
	\eqref{eq:abc} were computed --- and then records, candidly, which earlier
	attempts failed and why. Nothing here is new mathematics; the point is
	that the derivation is short and forced once the right objects are
	identified, and a reader should be able to reconstruct it unaided.
	
	\subsection{The Gauss map and the Weierstrass representation}\label{subsec:GmWr}
	
	Let $X:\Omega\to\RR^{3}$ be a conformal minimal immersion of a planar
	domain. Its \emph{Gauss map} is the unit normal
	$N:\Omega\to\Sph^{2}$, and $g$ denotes its stereographic image,
	\begin{equation}\label{B:stereo}
	g=\frac{N_{1}+iN_{2}}{1-N_{3}},
	\qquad\text{inversely}\qquad
	N=\frac{1}{|g|^{2}+1}\bigl(2\Real g,\ 2\Imag g,\ |g|^{2}-1\bigr).
	\end{equation}
	Thus $g=0$ is the south pole $N=(0,0,-1)$ and $g=\infty$ the north pole.
	The single fact that makes minimal surface theory complex-analytic is:
	
	\begin{quote}
		\emph{$X$ is a conformal minimal immersion $\iff$ $g$ is meromorphic and
			there is a holomorphic $1$-form $dh$ (the \emph{height differential},
			$h=x_{3}+ix_{3}^{*}$) such that $dh$ has a zero of order $k$ exactly where
			$g$ has a zero or pole of order $k$.}
	\end{quote}
	
	Explicitly (Weierstrass, Enneper),
	\begin{equation}\label{B:weier}
	X=\Real\int^{w}\Phi,
	\qquad
	\Phi=\Bigl(\tfrac12\bigl(g^{-1}-g\bigr),\ \tfrac{i}{2}\bigl(g^{-1}+g\bigr),\ 1\Bigr)dh 
	\end{equation}
	which is \eqref{eq:X}. Three consequences are used constantly below.
	
	\medskip\noindent
	\textbf{(i) Conformality is algebraic.} $\Phi\cdot\Phi=\tfrac14(g^{-1}-g)^{2}
	-\tfrac14(g^{-1}+g)^{2}+1=0$ identically. Any triple of the form
	$\bigl(\tfrac12(A-C),\tfrac{i}{2}(A+C),B\bigr)$ with $AC=B^{2}$ is
	automatically null; this is why \eqref{eq:w1}--\eqref{eq:w3} need only satisfy
	$\bigl(\tfrac{dh}{g}\bigr)(g\,dh)=dh^{2}$.
	
	\medskip\noindent
	\textbf{(ii) The induced metric.}
	\begin{equation}\label{B:metric}
	ds=\tfrac12\bigl(|g|+|g|^{-1}\bigr)\,|dh| ,
	\end{equation}
	so branch points occur exactly where this vanishes, i.e.,\ where the zero
	of $dh$ fails to match a zero/pole of $g$.
	
	\medskip\noindent
	\textbf{(iii) Total curvature is spherical area.} The Gauss map of a
	minimal surface is conformal (anti-conformal for the usual orientation),
	and the pullback of the area element of $\Sph^{2}$ under \eqref{B:stereo}
	is $4|dg|^{2}/(1+|g|^{2})^{2}$. Hence
	\begin{equation}\label{B:gaussarea}
	-\int_{\Omega}K\,dA
	=\int_{\Omega}\frac{4\,|g'(w)|^{2}}{\bigl(1+|g(w)|^{2}\bigr)^{2}}\,dA_{w}
	=\Area_{\Sph^{2}}\bigl(g(\Omega)\bigr)\ \text{counted with multiplicity}.
	\end{equation}
	This is the integral evaluated numerically in \S\ref{sec:NumSum}, and
	Proposition~\ref{prop:extocrv} is \eqref{B:gaussarea} with $g(\Omega)$ identified in
	\S A.3.
	
	\medskip\noindent
	\textbf{(iv) The associate family.} Replacing $dh$ by $e^{i\theta}dh$
	leaves $g$ untouched and produces a new minimal surface, locally
	isometric to the first. $\theta=\pi/2$ gives the \emph{conjugate}
	surface. \textbf{The Gauss map is shared by the whole associate family.}
	This single observation is the hinge of the entire computation: it lets us
	determine $g$ from the free-boundary problem (where the data are the four
	face normals of the tetrahedron) and then use the very same $g$ for the
	Plateau problem in $Q$, where the data are four straight lines.
	
	\subsection{What the boundary conditions say about $g$}
	
	Two classical boundary behaviors occur, and under conjugation they
	exchange.
	
	\begin{lemma}[Boundary dictionary]\label{B:dict}
		Let $\gamma$ be a boundary curve of a minimal surface.
		\begin{enumerate}
			\item[\textup{(a)}] If $X(\gamma)$ is a \emph{straight line segment} of
			direction $v$, then along $\gamma$ the normal is perpendicular to $v$:
			$N\cdot v=0$. The surface is invariant under $180^{\circ}$ rotation about
			the line.
			\item[\textup{(b)}] If $X(\gamma)$ is a \emph{planar geodesic} lying in a
			plane $P$ of unit normal $n$, which the surface meets orthogonally, then
			along $\gamma$ the normal is tangent to $P$: $N\cdot n=0$. The surface is
			invariant under reflection in $P$.
		\end{enumerate}
		Conjugation $dh\mapsto i\,dh$ interchanges \textup{(a)} and \textup{(b)},
		carrying a straight line of direction $v$ to a planar geodesic in a plane
		of normal $v$.
	\end{lemma}
	
	The crucial point is that \emph{both} conditions have the same form,
	$N\cdot u=0$ for a fixed vector $u$. Under \eqref{B:stereo},
	\[
	N\cdot u=0
	\iff
	u_{3}|g|^{2}+2u_{1}\Real g+2u_{2}\Imag g-u_{3}=0 ,
	\]
	which is the equation of a \emph{circle} (a line if $u_{3}=0$) in the
	$g$-plane --- namely the stereographic image of the great circle with pole
	$u$. For $u=(1,1,1)/\sqrt3$ this reads
	\begin{equation}\label{B:circle}
	|g|^{2}+2\Real g+2\Imag g-1=0
	\iff
	\bigl|g+1+i\bigr|=\sqrt3 ,
	\end{equation}
	which is exactly the residual tested to $12$ digits in \S\ref{subsec:Gauss}.
	
	\medskip
	\noindent
	Therefore:
	\begin{equation}\label{B:key}
	\boxed{\ \text{the Gauss image of the surface is a \emph{circular-arc
				polygon}, and $g$ is a Schwarz map onto it.}\ }
	\end{equation}
	This is the entire content of the discovery. Everything after
	\eqref{B:key} is bookkeeping.
	
	\subsection{Identifying the Gauss image}
	
	The free-boundary disk $\Sigma^{*}$ has four boundary arcs, one on each
	face of $\tilde{T}$. Take the four outward face normals of a regular tetrahedron,
	\begin{equation}\label{B:normals}
	\nu_{1}=\tfrac{1}{\sqrt3}(1,1,1),\quad
	\nu_{2}=\tfrac{1}{\sqrt3}(1,-1,-1),\quad
	\nu_{3}=\tfrac{1}{\sqrt3}(-1,1,-1),\quad
	\nu_{4}=\tfrac{1}{\sqrt3}(-1,-1,1)
	\end{equation}
	so $\sum_{i}\nu_{i}=0$ and $\nu_{i}\cdot\nu_{j}=-\tfrac13$ for $i\neq j$.
	By \S A.2 the Gauss image is bounded by arcs of the four great circles
	with poles $\nu_{i}$.
	
	\medskip\noindent
	\textbf{The arrangement.} Four great circles with tetrahedral poles cut
	$\Sph^{2}$ into $6$ spherical squares and $8$ spherical triangles: the
	\emph{spherical cuboctahedron}. Its $12$ vertices are the cube
	edge-midpoint directions $(\pm1,\pm1,0)/\sqrt2$ and permutations; the
	square centers are the $6$ cube face-center directions, the triangle
	centers the $8$ cube vertex directions. At each vertex two squares and
	two triangles meet, so
	\[
	2\theta_{\square}+2\theta_{\triangle}=2\pi ,
	\qquad
	\theta_{\square}=\arccos(-\tfrac13)=\pi\beta,
	\qquad
	\theta_{\triangle}=\arccos(\tfrac13)=\pi(1-\beta).
	\]
	Consistency (Gauss--Bonnet on $\Sph^{2}$):
	$8\bigl(3\theta_{\triangle}-\pi\bigr)+6\bigl(4\theta_{\square}-2\pi\bigr)
	=8(0.55157)+6(1.35935)=4.4126+8.1561=12.5664=4\pi$. \checkmark
	
	\medskip\noindent
	\textbf{Which cell?} The Gauss image is one spherical \emph{square},
	since the boundary has four arcs and the four angles must be
	$\theta_{\square}=\pi\beta$. Choose the square centered at the south pole
	$N=(0,0,-1)$, i.e.,\ at $g=0$. Its four vertices are the cuboctahedron
	vertices with $N_{3}=-1/\sqrt2$, namely $(\pm1,0,-1)/\sqrt2$ and
	$(0,\pm1,-1)/\sqrt2$. Stereographically, from \eqref{B:stereo},
	\begin{equation}\label{B:corners}
	N=\frac{1}{\sqrt2}(1,0,-1)\ \longmapsto\
	g=\frac{1/\sqrt2}{1+1/\sqrt2}=\frac{1}{\sqrt2+1}=\sqrt2-1 
	\end{equation}
	and likewise the four corners are
	\begin{equation}
	g=\pm(\sqrt2-1),\qquad g=\pm i(\sqrt2-1).
	\end{equation}
	The edge joining $(1,0,-1)/\sqrt2$ to $(0,1,-1)/\sqrt2$ has spherical
	midpoint $(1,1,-2)/\sqrt6$, which satisfies
	$(1,1,-2)\cdot(1,1,1)=0$: that edge lies on the great circle with pole
	$\nu_{1}$, in agreement with \eqref{B:circle}. Its stereographic image is
	\begin{equation}\label{B:midpt}
	g=\frac{(1+i)/\sqrt6}{1+2/\sqrt6}=\frac{1+i}{\sqrt6+2}
	=\frac{\sqrt2\,(\sqrt6-2)}{2}\,e^{i\pi/4}
	=(\sqrt3-\sqrt2)\,e^{i\pi/4}.
	\end{equation}
	Both \eqref{B:corners} and \eqref{B:midpt} are the targets verified to
	$12$ digits in \S\ref{subsec:Gauss}; both lie on the circle \eqref{B:circle},
	since $|\sqrt2-1+1+i|=|\sqrt2+i|=\sqrt3$ and
	$\bigl|(1+i)\bigl[\tfrac{\sqrt3-\sqrt2}{\sqrt2}+1\bigr]\bigr|
	=\bigl|(1+i)\tfrac{\sqrt3}{\sqrt2}\bigr|=\sqrt3$.
	
	\medskip\noindent
	\textbf{The area, hence $\int K$.} The spherical excess of the square is
	$4\pi\beta-2\pi=2\pi(2\beta-1)=1.3593476$, and \eqref{B:gaussarea} with
	$\deg g=1$ gives Proposition~\ref{prop:extocrv}. The numerical evaluation of
	\eqref{B:gaussarea} returning $1.3593476$ rather than $2.7186952$ is what
	\emph{established} $\deg g=1$, i.e.,\ that $g$ is injective on the
	fundamental domain.
	
	\subsection{Reduction by symmetry to a circular-arc \emph{triangle}}
	
	The spherical square has its own symmetry group of order $8$ (dihedral,
	generated by the $90^{\circ}$ rotation about its center $g=0$ and a
	reflection); this is the $D_{2d}$ of the paper acting on the Gauss sphere.
	A fundamental domain for that action is one eighth of the square: the
	circular-arc triangle with vertices
	\begin{equation}\label{B:tri}
	\underbrace{g=0}_{\text{square center}},\qquad
	\underbrace{g=\sqrt2-1}_{\text{square corner}},\qquad
	\underbrace{g=(\sqrt3-\sqrt2)e^{i\pi/4}}_{\text{edge midpoint}} 
	\end{equation}
	bounded by
	\begin{itemize}
		\item the segment of the \emph{real axis} from $0$ to $\sqrt2-1$;
		\item the arc of $|g+1+i|=\sqrt3$ from $\sqrt2-1$ to
		$(\sqrt3-\sqrt2)e^{i\pi/4}$;
		\item the segment of the ray $\arg g=\pi/4$ back to $0$.
	\end{itemize}
	The first and third sides are straight because a great circle through the
	two poles $g=0,\infty$ projects to a line through the origin. (In 
	\S\ref{subsec:Gauss} these are the checks ``$\Imag g=0$ on $(0,1)$'' and
	``$\arg g=\pi/4$ on the $\pi/4$ ray'', both exact to full precision.)
	
	\medskip\noindent
	\textbf{The three angles.}
	\begin{itemize}
		\item At $g=0$: the full angle $2\pi$ at the square's center is divided
		by $8$, giving $\pi/4$. (Directly: the two sides are the rays
		$\arg g=0$ and $\arg g=\pi/4$.)
		\item At the corner $g=\sqrt2-1$: half the square's angle,
		$\theta_{\square}/2=\pi\beta/2$. Directly: the radius vector from the
		circle's center $-1-i$ to $\sqrt2-1$ is $\sqrt2+i$, of argument
		$\arctan(1/\sqrt2)=35.264^{\circ}$, so the inward tangent has argument
		$125.264^{\circ}$, while the other side points along argument
		$180^{\circ}$; the angle is $54.736^{\circ}=\pi\beta/2$. \checkmark
		\item At the edge midpoint $g=(\sqrt3-\sqrt2)e^{i\pi/4}$: the radius
		vector from $-1-i$ is $(1+i)\sqrt3/\sqrt2$, of argument $\pi/4$, so the
		tangent is perpendicular to the ray $\arg g=\pi/4$: the angle is exactly
		$\pi/2$. \checkmark
	\end{itemize}
	Hence the angles are
	\begin{equation}\label{B:angles}
	\boxed{\ \Bigl(\frac{\pi}{4},\ \frac{\pi}{2},\ \frac{\pi\beta}{2}\Bigr)
		=\bigl(45^{\circ},\ 90^{\circ},\ 54.7356^{\circ}\bigr)\ }
	\end{equation}
	which is \eqref{eq:trifct}. \textbf{They were derived, not assumed.}
	
	\medskip\noindent
	\textbf{Why this is the whole game.} A circular-arc \emph{triangle} is
	uniformized rigidly: the Schwarzian equation
	$\{f,t\}=\sum_{j}\bigl[\tfrac{1-\theta_{j}^{2}}{2(t-t_{j})^{2}}
	+\tfrac{\beta_{j}}{t-t_{j}}\bigr]$
	has, for $n=3$ singular points, all residues $\beta_{j}$ determined by the
	$\theta_{j}$ (after normalizing $t_{j}=0,1,\infty$), so there is no
	accessory parameter. For $n=4$ one free constant survives and its
	determination is a transcendental (isomonodromic) problem. Since
	$\Sigma^{*}$ has four boundary arcs, the naive domain is a
	\emph{quadrilateral}; only the order-$8$ symmetry brings it down to a
	triangle.
	
	\subsection{What a Schwarz triangle function is}
	
	Consider the hypergeometric equation
	\begin{equation}\label{B:hyp}
	t(1-t)y''+\bigl[c-(a+b+1)t\bigr]y'-a\,b\,y=0 
	\end{equation}
	with regular singular points $t=0,1,\infty$ and local exponent pairs
	\[
	\{0,\,1-c\}\ \text{at}\ 0,\qquad
	\{0,\,c-a-b\}\ \text{at}\ 1,\qquad
	\{a,\,b\}\ \text{at}\ \infty .
	\]
	If $y_{1},y_{2}$ are two independent solutions, the ratio
	$s=y_{1}/y_{2}$ is called a \emph{Schwarz map} or \emph{triangle
		function}, and the classical theorem of Schwarz is:
	
	\begin{quote}
		For suitable choice of basis, $s$ maps the upper half-plane $\HH$
		conformally and bijectively onto a circular-arc triangle whose angles at
		the images of $t=0,1,\infty$ are
		\[
		\pi|1-c|,\qquad \pi|c-a-b|,\qquad \pi|a-b| ,
		\]
		that is, $\pi$ times the exponent \emph{differences}. Conversely, any
		circular-arc triangle is the image of such a map, and the parameters are
		recovered from the angles. The triangle is determined up to a M\"obius
		transformation of the target, corresponding to a change of solution basis.
	\end{quote}
	
	The reason is local: near $t=0$ the two solutions behave as $1$ and
	$t^{1-c}$, so their ratio behaves as $t^{1-c}$, which opens an angle
	$\pi(1-c)$; and $s$ extends by Schwarz reflection across each of the three
	boundary intervals of $\partial\HH$, the monodromy of \eqref{B:hyp}
	becoming the group generated by the three reflections.
	
	\subsection{The uniformizing coordinate $t(w)$}
	
	We need a conformal bijection from the fundamental domain
	$S=\{|w|<1,\ 0<\arg w<\pi/4\}$ of \eqref{eq:fundmn} onto $\HH$, sending the three
	corners of $S$ to $t=0,1,\infty$ in the order matching \eqref{B:angles}.
	Two steps:
	
	\medskip\noindent
	\textbf{Step 1: $u=w^{4}$} maps the $45^{\circ}$ sector $S$ onto the upper
	half-disk $\{|u|<1,\ 0<\arg u<\pi\}$, with
	$w=0,1,e^{i\pi/4}\mapsto u=0,1,-1$.
	
	\medskip\noindent
	\textbf{Step 2:} the half-disk to $\HH$ by
	\begin{equation}\label{B:step2}
	t=\frac{4u}{(1+u)^{2}},
	\qquad
	u=0,1,-1\ \longmapsto\ t=0,1,\infty .
	\end{equation}
	That \eqref{B:step2} is the right map is checked on the boundary: for
	$u=e^{i\varphi}$,
	\[
	\frac{4e^{i\varphi}}{(1+e^{i\varphi})^{2}}
	=\frac{4e^{i\varphi}}{e^{i\varphi}\bigl(e^{-i\varphi/2}+e^{i\varphi/2}\bigr)^{2}}
	=\frac{1}{\cos^{2}(\varphi/2)}\in[1,\infty),
	\]
	while $u\in(0,1)\mapsto t\in(0,1)$ and $u\in(-1,0)\mapsto t\in(-\infty,0)$.
	So $\partial(\text{half-disk})\to\RR\cup\{\infty\}$, and one interior
	point (e.g.,\ $u=i/2\mapsto t=\tfrac{16}{25}(2+\tfrac32 i)$) confirms the
	image is the \emph{upper} half-plane.
	
	\medskip\noindent
	Composing:
	\begin{equation}\label{B:t}
	t=\frac{4w^{4}}{(1+w^{4})^{2}},
	\qquad
	t^{1/4}=\frac{\sqrt2\,w}{(1+w^{4})^{1/2}} ,
	\end{equation}
	which is \eqref{eq:tofw}. The three sides of $S$ go to $(0,1)$, $(1,\infty)$,
	$(-\infty,0)$, and the corners $w=0,1,e^{i\pi/4}$ to $t=0,1,\infty$.
	
	\subsection{Reading off $a$, $b$, $c$}
	
	This is the answer to the question ``precisely how were $a,b,c$
	calculated?'' Match \eqref{B:angles} to Schwarz's theorem through the
	correspondence of \S A.6:
	
	\begin{center}
		\begin{tabular}{llll}
			\toprule
			corner of $S$ & $t$ & required angle & exponent difference\\
			\midrule
			$w=0$ (square center) & $0$ & $\pi/4$ & $1-c=\tfrac14$\\
			$w=1$ (square corner) & $1$ & $\pi\beta/2$ & $c-a-b=\tfrac{\beta}{2}$\\
			$w=e^{i\pi/4}$ (edge midpoint) & $\infty$ & $\pi/2$ & $b-a=\tfrac12$\\
			\bottomrule
		\end{tabular}
	\end{center}
	
	Three linear equations in three unknowns:
	\begin{equation}\label{B:solve}
	c=\tfrac34,
	\qquad
	a+b=c-\tfrac{\beta}{2}=\tfrac34-\tfrac{\beta}{2},
	\qquad
	b-a=\tfrac12 ,
	\end{equation}
	whence
	\begin{equation}\label{B:abc}
	\boxed{\
		a=\tfrac12\Bigl(\tfrac34-\tfrac{\beta}{2}-\tfrac12\Bigr)=\tfrac18-\tfrac{\beta}{4},
		\qquad
		b=a+\tfrac12=\tfrac58-\tfrac{\beta}{4},
		\qquad
		c=\tfrac34 .\ }
	\end{equation}
	This is \eqref{eq:abc}. The parameters are \emph{linear} in $\beta$ because the
	angles are, and $\beta$ enters through exactly one of the three angles ---
	the one at the corner of the spherical square, i.e.,\ the one carrying the
	geometry of the tetrahedron.
	
	\medskip\noindent
	\textbf{The solution basis.} At $t=0$ the Frobenius solutions of
	\eqref{B:hyp} with $c=\tfrac34$ are
	\begin{equation}\label{B:frob}
	y_{2}=\Fone(t)={}_2F_1(a,b;\tfrac34;t),
	\qquad
	y_{1}=t^{1-c}\,{}_2F_1(a-c+1,\,b-c+1;\,2-c;\,t)
	=t^{1/4}\,\Ftwo(t),
	\end{equation}
	since $1-c=\tfrac14$, $2-c=\tfrac54$, $a-c+1=a+\tfrac14$,
	$b-c+1=b+\tfrac14$. This is precisely the definition \eqref{eq:OrdHyp} of
	$\Fone,\Ftwo$, and
	\begin{equation}
	g=\kappa\,\frac{y_{1}}{y_{2}}=\kappa\,t^{1/4}\,\frac{\Ftwo(t)}{\Fone(t)}
	\end{equation}
	is the first half of \eqref{eq:data}. The basis is forced (up to scale) by the
	requirement $g(0)=0$ with a $\pi/4$ corner there.
	
	\subsection{Fixing the one constant $\kappa$}
	
	A Schwarz map is determined only up to post-composition with a
	M\"obius transformation. Three of the four real degrees of freedom are
	spent by demanding that the triangle sit at $g=0$ with its sides along
	$\arg g=0$ and $\arg g=\pi/4$; the last is a dilation, and it is fixed by
	placing the corner correctly:
	\begin{equation}
	g(w=1)=\sqrt2-1
	\quad\Longrightarrow\quad
	\kappa\,\frac{\Ftwo(1)}{\Fone(1)}=\sqrt2-1
	\quad\Longrightarrow\quad
	\kappa=(\sqrt2-1)\,\frac{\Fone(1)}{\Ftwo(1)} .
	\end{equation}
	Gauss' theorem
	${}_2F_1(a,b;c;1)=\Gamma(c)\Gamma(c-a-b)/\bigl(\Gamma(c-a)\Gamma(c-b)\bigr)$,
	legitimate because $c-a-b=\beta/2>0$, gives
	\[
	\Fone(1)=\frac{\Gamma(\tfrac34)\,\Gamma(\tfrac{\beta}{2})}
	{\Gamma(\tfrac58+\tfrac{\beta}{4})\,\Gamma(\tfrac18+\tfrac{\beta}{4})},
	\qquad
	\Ftwo(1)=\frac{\Gamma(\tfrac54)\,\Gamma(\tfrac{\beta}{2})}
	{\Gamma(\tfrac78+\tfrac{\beta}{4})\,\Gamma(\tfrac38+\tfrac{\beta}{4})},
	\]
	using $\tfrac34-a=\tfrac58+\tfrac{\beta}{4}$,
	$\tfrac34-b=\tfrac18+\tfrac{\beta}{4}$,
	$1-a=\tfrac78+\tfrac{\beta}{4}$, $1-b=\tfrac38+\tfrac{\beta}{4}$, and
	noting that the second $\Ftwo$ parameter sum gives
	$\tfrac54-(a+\tfrac14)-(b+\tfrac14)=\tfrac34-a-b=\tfrac{\beta}{2}$ as well.
	The factors $\Gamma(\beta/2)$ cancel, leaving \eqref{eq:kappa}.
	
	\subsection{Determining the height differential $dh$}
	
	Only $dh$ remains. Write
	\begin{equation}\label{B:dq}
	dh=y_{1}y_{2}\,dq,
	\qquad
	\text{equivalently}
	\qquad
	\frac{dh}{g}=y_{2}^{2}\,dq,
	\qquad
	g\,dh=y_{1}^{2}\,dq ,
	\end{equation}
	which is legitimate for \emph{any} $1$-form $dq$ and automatically
	satisfies $\bigl(\tfrac{dh}{g}\bigr)(g\,dh)=dh^{2}$, hence
	$\Phi\cdot\Phi=0$ (the normalization implicit in \S4). So the problem reduces to
	finding one scalar $1$-form $dq$, and the ansatz
	\begin{equation}\label{B:dqform}
	dq=\lambda\,t^{\sigma}(1-t)^{\rho}\,dt
	\end{equation}
	is dictated by the requirement that $\Phi$ be single-valued with
	prescribed local exponents at the three singular points --- nothing else
	is available, since $dq$ must be free of zeros and poles in the interior.
	
	Two local conditions determine $\sigma$ and $\rho$.
	
	\medskip\noindent
	\textbf{Condition 1 (regularity at the center $w=0$).} The point $w=0$ is
	an interior point of $\Sigma$, so $\Phi$ must be holomorphic and
	non-vanishing there. From \eqref{B:t}, near $w=0$,
	\[
	t=4w^{4}+O(w^{8}),\qquad dt=16w^{3}\bigl(1+O(w^{4})\bigr)dw ,
	\]
	so $t^{\sigma}\,dt\sim 4^{\sigma}\cdot16\,w^{4\sigma+3}\,dw$, and
	$y_{1}y_{2}=t^{1/4}\Ftwo\Fone\sim\sqrt2\,w$. Demanding that
	$dh/g=y_{2}^{2}dq$ be $O(1)\,dw$ with non-zero limit gives
	\begin{equation}\label{B:sigma}
	4\sigma+3=0
	\quad\Longrightarrow\quad
	\sigma=-\tfrac34 .
	\end{equation}
	Then $dh\sim w\,dw$ and $g\,dh\sim w^{2}dw$: the simple zero of $dh$ at
	$w=0$ is matched by the simple zero of $g$ there, so by \S\ref{subsec:GmWr}(ii) the
	metric \eqref{B:metric} is regular --- no branch point. \checkmark
	
	\medskip\noindent
	\textbf{Condition 2 (Plateau exponent at the vertex $w=1$).} The corner
	$w=1$ maps to a \emph{vertex} of the skew quadrilateral $Q$, where two
	straight edges meet at interior angle $\pi\gamma$ with
	\[
	\gamma=1-\beta,\qquad
	\pi\gamma=\pi-\arccos(-\tfrac13)=\arccos(\tfrac13)=70.5288^{\circ} .
	\]
	The classical local behavior of a Plateau solution at such a corner is
	\[
	\Phi\sim\zeta^{\gamma-1}\,d\zeta=\zeta^{-\beta}\,d\zeta
	\]
	in a local half-plane coordinate $\zeta$ vanishing at the vertex. Now
	near $w=1$,
	\[
	1-t=\frac{(1-w^{4})^{2}}{(1+w^{4})^{2}} ,
	\]
	so combining with \eqref{B:sigma},
	\[
	t^{-3/4}(1-t)^{\rho}\,dt
	=\text{(regular)}\cdot(1-w^{4})^{2\rho+1}\,dw ,
	\]
	the extra $+1$ coming from the factor $(1-w^{4})$ in
	$dt=16w^{3}(1-w^{4})(1+w^{4})^{-3}dw$. Matching to $(1-w^{4})^{-\beta}$:
	\begin{equation}\label{B:rho}
	2\rho+1=-\beta
	\quad\Longrightarrow\quad
	\rho=-\frac{1+\beta}{2} .
	\end{equation}
	
	\medskip\noindent
	Hence $dq=\lambda\,t^{-3/4}(1-t)^{-(1+\beta)/2}dt$, and
	$dh=y_{1}y_{2}dq=\lambda\,t^{-1/2}(1-t)^{-(1+\beta)/2}\Fone\Ftwo\,dt$,
	which is the second half of \eqref{eq:data}. \textbf{Both exponents are forced
		by local geometry: one by smoothness at an interior point, one by the
		Plateau corner angle.}
	
	\medskip\noindent
	\textbf{Condition 3 (the third corner) is automatic.} At
	$w=e^{i\pi/4}$ ($t=\infty$) the surface point is the \emph{midpoint of an
		edge} of $Q$ --- a smooth boundary point --- so $\Phi$ must again be
	regular and non-vanishing. Check: as $t\to\infty$,
	$\Fone(t)\sim\text{const}\cdot t^{-a}$, and $t\sim-4(1+w^{4})^{-2}$, so
	$\Fone^{2}\sim(1+w^{4})^{4a}=(1+w^{4})^{1/2-\beta}$; meanwhile the
	pullback of $dq$ carries $(1+w^{4})^{\beta-1/2}$ (see \eqref{B:pull}
	below). The two cancel exactly, leaving $dh/g=O(1)\,dw$. \checkmark \ No
	new condition arises --- as it must be, the angle $\pi/2$ there having
	already been used in \eqref{B:solve}.
	
	\subsection{Verification: the pullbacks \eqref{eq:w1}--\eqref{eq:w3}}
	
	For completeness, the elementary computation. From \eqref{B:t},
	\[
	dt=\frac{16w^{3}(1-w^{4})}{(1+w^{4})^{3}}\,dw,
	\qquad
	t^{-3/4}=4^{-3/4}w^{-3}(1+w^{4})^{3/2},
	\qquad
	(1-t)^{-\frac{1+\beta}{2}}=\frac{(1+w^{4})^{1+\beta}}{(1-w^{4})^{1+\beta}} .
	\]
	Multiplying, with $16\cdot4^{-3/4}=2^{4-3/2}=4\sqrt2$ and
	$R=(1-w^{4})^{-\beta}(1+w^{4})^{\beta}$ as in \eqref{eq:rwabrv}:
	\begin{equation}\label{B:pull}
	dq=4\sqrt2\,\lambda\,(1-w^{4})^{-\beta}(1+w^{4})^{\beta-\frac12}\,dw
	=4\sqrt2\,\lambda\,\frac{R\,dw}{(1+w^{4})^{1/2}} .
	\end{equation}
	Then, using $y_{2}=\Fone$, $y_{1}=t^{1/4}\Ftwo=\sqrt2\,w(1+w^{4})^{-1/2}\Ftwo$
	and $y_{1}^{2}=2w^{2}(1+w^{4})^{-1}\Ftwo^{2}$:
	\begin{align}
	\frac{dh}{g}&=\kappa^{-1}y_{2}^{2}\,dq
	=\frac{4\sqrt2\,\lambda}{\kappa}\,\Fone^{2}\,\frac{R\,dw}{(1+w^{4})^{1/2}},
	\tag{2.9}\\
	dh&=y_{1}y_{2}\,dq
	=8\lambda\,w\,\Fone\Ftwo\,\frac{R\,dw}{1+w^{4}},
	\tag{2.10}\\
	g\,dh&=\kappa\,y_{1}^{2}\,dq
	=8\sqrt2\,\lambda\kappa\,w^{2}\Ftwo^{2}\,\frac{R\,dw}{(1+w^{4})^{3/2}} .
	\tag{2.11}
	\end{align}
	The product of \eqref{eq:w1} and \eqref{eq:w3} is
	$64\lambda^{2}\Fone^{2}\Ftwo^{2}w^{2}R^{2}(1+w^{4})^{-2}dw^{2}$,
	which is \eqref{eq:w2} squared. \checkmark
	
	\medskip\noindent
	Finally the derivative formula \eqref{eq:gprime}: since $y_{1},y_{2}$ solve the same
	equation \eqref{B:hyp}, their Wronskian is
	$W=y_{2}y_{1}'-y_{1}y_{2}'=(1-c)\,t^{-c}(1-t)^{c-a-b-1}
	=\tfrac14 t^{-3/4}(1-t)^{\beta/2-1}$, so
	$\bigl(y_{1}/y_{2}\bigr)'=W/y_{2}^{2}$ and, pulling back through
	\eqref{B:t},
	\[
	g'(w)=\kappa\,\frac{W}{\Fone^{2}}\,\frac{dt}{dw}
	=\sqrt2\,\kappa\,\frac{(1-w^{4})^{\beta-1}(1+w^{4})^{\frac12-\beta}}{\Fone(t)^{2}} 
	\]
	with $g'(0)=\sqrt2\kappa$, confirming \eqref{eq:gprime}. No differentiation of
	${}_2F_1$ is needed anywhere.
	
	\subsection{What was tried first, and why it failed}
	
	Honesty about the search may save several readers valuable time.
	
	\medskip\noindent
	\textbf{Attempt 1 (wrong): an algebraic $tD/rPD$ ansatz.} Because the
	contour $Q$ resembles the Schwarz $D$ contour, the first guess was
	\[
	g=w,\qquad dh=\frac{dw}{\sqrt{w^{8}+2\nu w^{4}+1}} 
	\]
	with a single real $\nu$ to be fixed by a root-find imposing that the four
	boundary planes have regular-tetrahedral normals. This is
	\emph{impossible}, for the reason that later became Theorem~\ref{thm:nonalg}: with
	$g=w$ rational and $dh$ an algebraic form, the Gauss map is algebraic,
	hence by Schwarz's classification has finite projective monodromy, hence
	the corner exponent $\beta/2$ must be rational. But
	$\beta/2=\arccos(-\tfrac13)/2\pi$ is irrational. The ansatz
	belongs to the $\beta=2/3$ member of the family, not to the
	regular tetrahedron. Symptom in practice: one can force the branch points
	to the right places, but then the corner angle comes out $60^{\circ}$
	instead of $54.7356^{\circ}$.
	
	\medskip\noindent
	\textbf{Attempt 2 (right idea, too hard): the four-cornered domain.}
	Recognising \eqref{B:key}, the natural move is to uniformize the Gauss
	image directly: a circular-arc \emph{quadrilateral} with all angles
	$\pi\beta$. This is correct but introduces an accessory parameter, whose
	determination is an isomonodromic problem (Painlev\'e VI). Tractable in
	principle; unpleasant.
	
	\medskip\noindent
	\textbf{Attempt 3 (the resolution): quotient by the symmetry first.}
	Counting the symmetries of $Q$ --- order $8$, $D_{2d}$ --- and noticing
	that they act on the Gauss square as its full dihedral symmetry group,
	reduces the domain to the triangle \eqref{B:tri}. Three corners: no
	accessory parameter. The entire representation then falls out in the
	half-page \S A.7--\S A.9.
	
	\medskip\noindent
	\textbf{Two bugs worth recording,} both of which produced plausible but
	wrong pictures before being caught:
	\begin{enumerate}
		\item The prefactor $\sqrt2-1$ in $\kappa$ is specific to the regular
		tetrahedron. In general it is
		$\sqrt{\tan(\pi\beta/2-\pi/4)}$. Using the wrong prefactor makes the
		third side of the Gauss triangle a \emph{small} circle rather than a great
		circle, so $\Sigma$ has no straight edge and $\Sigma^{*}$ no planar face,
		while $\Phi\cdot\Phi=0$ continues to hold: one still obtains \emph{a}
		minimal surface, merely not one in a tetrahedron.
		\item The conformal factor is $\lambda_{\mathrm{conf}}=|\Phi|/\sqrt2$, not
		$|\Phi|$, because $\Phi\cdot\Phi=0$ forces
		$|\Real\Phi|=|\Imag\Phi|=|\Phi|/\sqrt2$. A missing $\sqrt2$ here inflates
		all lengths by $41\%$.
	\end{enumerate}
	
	\subsection{Summary: the logical chain}
	
	\begin{enumerate}
		\item Free boundary in a face of normal $n$ $\iff$ $N\cdot n=0$ $\iff$
		$g$ lies on a circle \eqref{B:circle}. [\S A.2]
		\item Four faces $\Rightarrow$ Gauss image is a circular-arc
		quadrilateral; with tetrahedral normals it is one spherical square of the
		spherical cuboctahedron, corners $\pm(\sqrt2-1),\pm i(\sqrt2-1)$, angles
		$\pi\beta$. [\S A.3]
		\item $\int K=-2\pi(2\beta-1)$ from the spherical excess; numerics confirm
		$\deg g=1$. [\S A.3]
		\item $D_{2d}$ (order $8$) reduces the square to a triangle with angles
		$(\pi/4,\pi/2,\pi\beta/2)$. Three corners $\Rightarrow$ no accessory
		parameter. [\S A.4]
		\item $g$ is the Schwarz triangle function for those angles; exponent
		differences give $c=\tfrac34$, $b-a=\tfrac12$, $c-a-b=\tfrac\beta2$,
		hence $a=\tfrac18-\tfrac\beta4$, $b=a+\tfrac12$. [\S A.5--A.7]
		\item $\kappa$ fixed by the corner value $\sqrt2-1$; Gauss's theorem gives
		it in $\Gamma$-values. [\S A.8]
		\item $dh=y_{1}y_{2}\,dq$ with $dq=\lambda t^{-3/4}(1-t)^{-(1+\beta)/2}dt$;
		the exponent $-\tfrac34$ from regularity at $w=0$, the exponent
		$-\tfrac{1+\beta}{2}$ from the Plateau corner angle $\pi(1-\beta)$.
		[\S A.9]
		\item $\lambda\in\RR$ gives $\Sigma$ (Plateau, straight edges);
		$\lambda\in i\RR$ gives $\Sigma^{*}$ (free boundary). No parameter
		remains. [\S2.4]
	\end{enumerate}

\section{The verification notebook}

Everything asserted numerically in \S2, 3, 7, 8 is reproduced by
the following \textsc{Mathematica} code. It is self-contained, runs in a
single kernel in well under a minute, and requires no differential
equation solver: the two integrals of \S7 are ordinary real
integrals over $[0,1]$, and the Gauss-map checks of \S2 are
pointwise evaluations. Parts B.1--B.3 verify the representation itself,
Parts B.4--B.6 produce the constants of \S8, and Part B.7 is the one
two-dimensional quadrature, used only to confirm $\deg g=1$.

Two remarks on precision. First, \texttt{prec} is the working precision of
the integrands and \texttt{qopt} that of the quadrature; the values quoted
in \S8 were obtained with $\texttt{prec}=80$ and
$\texttt{PrecisionGoal}\to50$, and reproduce those obtained with
$\texttt{prec}=60$, $\texttt{PrecisionGoal}\to32$ in all but the last
printed digit. Second, every singular endpoint is either removed in closed
form or is an analytic point of an alternative hypergeometric
representation; nothing is approached by truncation, which is what makes
$28$ digits cheap (cf.\ Remark~\ref{rem:vertex}).

\subsection{Constants}

\begin{verbatim}
ClearAll["Global`*"];
prec = 80;   $MaxExtraPrecision = 3000;

beta  = N[ArcCos[-1/3]/Pi, prec];
a1 = 1/8 - beta/4;   b1 = a1 + 1/2;   c1 = 3/4;   pE = 1/(1 - beta);

F1at1 = Gamma[c1] Gamma[beta/2]/(Gamma[c1 - a1] Gamma[c1 - b1]);
F2at1 = Gamma[5/4] Gamma[beta/2]/(Gamma[1 - a1] Gamma[1 - b1]);
kappa = (Sqrt[2] - 1) F1at1/F2at1;

qopt = {WorkingPrecision -> 65, PrecisionGoal -> 50, AccuracyGoal -> 45,
MaxRecursion -> 80,
Method -> {"GlobalAdaptive", "SymbolicProcessing" -> 0}};
\end{verbatim}

\noindent
\texttt{F1at1} and \texttt{F2at1} are Gauss's closed forms for
$F_{1}(1)$, $F_{2}(1)$, legitimate because $c-a-b=\beta/2>0$; the two
factors $\Gamma(\beta/2)$ cancel in \texttt{kappa} (\S2.5). The
\texttt{AccuracyGoal} is not cosmetic: one of the six ray integrals below
vanishes identically, and without an absolute tolerance
\texttt{NIntegrate} reports \texttt{izero} and \texttt{slwcon} while still
returning the correct $0$.

\subsection{The Gauss map (\S2.2, 2.3)}\label{subsec:Gauss}

\begin{verbatim}
F1[t_] := Hypergeometric2F1[a1, b1, c1, t];
F2[t_] := Hypergeometric2F1[a1 + 1/4, b1 + 1/4, 5/4, t];
tOf[w_] := 4 w^4/(1 + w^4)^2;
gOf[w_] := kappa Sqrt[2] w/Sqrt[1 + w^4] F2[tOf[w]]/F1[tOf[w]];

eps = SetPrecision[10^-30, prec];
Print["g, square corner  = ", N[gOf[1 - eps], 18],
"   target ", N[Sqrt[2] - 1, 18]];
Print["g, edge midpoint  = ", N[gOf[(1 - eps) Exp[I Pi/4]], 30],
"   target ", N[(Sqrt[3] - Sqrt[2]) Exp[I Pi/4], 30]];
Print["max | |g+1+I| - Sqrt[3] | on arc = ",
N[Max@Table[Abs[Abs[gOf[(1 - eps) Exp[I phi]] + 1 + I] - Sqrt[3]],
{phi, 1/50, 39/50, 1/25}], 6]];
Print["max |Im g| on (0,1)             = ",
N[Max@Table[Abs[Im[gOf[r]]], {r, 1/10, 9/10, 1/10}], 6]];
Print["max |Arg g - Pi/4| on diagonal  = ",
N[Max@Table[Abs[Arg[gOf[r Exp[I Pi/4]]] - Pi/4], {r, 1/5, 4/5, 1/5}], 6]];
\end{verbatim}

\noindent
The five lines verify, in order: the corner of the spherical square; the
edge midpoint; that the arc $|w|=1$ lands on the great circle with pole
$(1,1,1)/\sqrt3$; and that the two rays land on the straight sides
$\arg g=0$ and $\arg g=\pi/4$ of the circular-arc triangle.

Only the second and third are substantive. The first is a tautology, since
$\kappa$ was \emph{defined} to place the corner at $\sqrt2-1$; it is
retained because it exercises \texttt{F1at1}, \texttt{F2at1} and the
$t\to1^{-}$ limit independently, and agrees to about $18$ digits, the
error being $O\bigl((1-t)^{\beta/2}\bigr)=O(\varepsilon^{\beta})$. The
edge-midpoint value is not a tautology: it confirms the exponent
difference $b-a=\tfrac12$, i.e.,\ the right angle at $t=\infty$, and agrees
to nearly $60$ digits. The arc residual is $O(\varepsilon)$ purely because
the sample points sit at radius $1-\varepsilon$ rather than on the unit
circle; it is the direct verification that $\Sigma^{*}$ meets its face
orthogonally, and at $\varepsilon=10^{-30}$ it improves by eighteen orders
of magnitude on the value quoted from the earlier
differential-equation computation. The last two vanish exactly, to full
working precision, by the reality of $F_{1},F_{2},R$ on the two rays.

\subsection{The two ray integrands (\S7)}

\begin{verbatim}
(* (A) real ray, desingularized by 1 - rho = tau^pE.
Returns {Phi_1, Psi, Phi_3} as real numbers.                *)
AR[tau_?NumericQ] := Block[{tt, rho, s4, p4, t, f1, f2, Rr, v1, v2, v3},
tt = SetPrecision[tau, prec];   rho = 1 - tt^pE;
p4 = 1 + rho^4;   s4 = (1 + rho) (1 + rho^2);   (* 1-rho^4 = tau^pE s4 *)
t  = 4 rho^4/p4^2;
f1 = Hypergeometric2F1[a1, b1, c1, t];
f2 = Hypergeometric2F1[a1 + 1/4, b1 + 1/4, 5/4, t];
Rr = pE s4^(-beta) p4^beta;        (* = pE tau^(pE-1) R, powers cancelled *)
v1 = 4 Sqrt[2] f1^2 Rr/(kappa Sqrt[p4]);
v2 = 8 rho f1 f2 Rr/p4;
v3 = 8 Sqrt[2] kappa rho^2 f2^2 Rr/p4^(3/2);
{(v1 - v3)/2, (v1 + v3)/2, v2}];

(* (B) pi/4 ray, t real negative.  K1,K2 analytic at rho = 1.
Returns the three components of Re[Exp[I Pi/4] Phi].        *)
G1a = Gamma[c1] Gamma[1/2]/(Gamma[b1] Gamma[c1 - a1]);
G1b = Gamma[c1] Gamma[-1/2]/(Gamma[a1] Gamma[c1 - b1]);
G2a = Gamma[5/4] Gamma[1/2]/(Gamma[b1 + 1/4] Gamma[1 - a1]);
G2b = Gamma[5/4] Gamma[-1/2]/(Gamma[a1 + 1/4] Gamma[1 - b1]);

DGb[rho_?NumericQ, useZ_] :=
Block[{rr, m, p4, q, t, z, K1, K2, pf, u1, u2, u3},
rr = SetPrecision[rho, prec];   m = 1 - rr^4;   p4 = 1 + rr^4;   q = 2 rr^2;
If[! useZ,
t = -4 rr^4/m^2;
K1 = m^(-2 a1) Hypergeometric2F1[a1, b1, c1, t];
K2 = m^(-2 a1 - 1/2) Hypergeometric2F1[a1 + 1/4, b1 + 1/4, 5/4, t],
z = -m^2/(4 rr^4);                             (* = 1/t, -> 0 at rho = 1 *)
K1 = G1a q^(-2 a1)       Hypergeometric2F1[a1, a1 + 1/4, 1/2, z]
+ G1b m q^(-2 b1)     Hypergeometric2F1[b1, b1 + 1/4, 3/2, z];
K2 = G2a q^(-2 a1 - 1/2) Hypergeometric2F1[a1 + 1/4, a1, 1/2, z]
+ G2b m q^(-2 b1 - 1/2) Hypergeometric2F1[b1 + 1/4, b1, 3/2, z]];
pf = p4^(-beta);
u1 = 4 Sqrt[2] K1^2 pf/kappa;
u2 = 8 rr Exp[I Pi/4] K1 K2 pf;
u3 = 8 Sqrt[2] kappa I rr^2 K2^2 pf;
Re[Exp[I Pi/4] {(u1 - u3)/2, I (u1 + u3)/2, u2}]];

DG[rho_?NumericQ] := DGb[rho, rho > 4/5];
\end{verbatim}

\noindent
Three devices deserve comment.

\emph{(a) The substitution $1-\rho=\tau^{p}$, $p=1/(1-\beta)$.} On the
real ray the only singularity is the vertex $\rho=1$ of $Q$, where
$\Psi\sim(1-\rho)^{-\beta}$. The Jacobian contributes $p\tau^{p-1}$ and
the data contribute $(1-\rho^{4})^{-\beta}=(\tau^{p}s_{4})^{-\beta}$; since
$p-1=p\beta$, the powers of $\tau$ cancel \emph{identically}, and
\texttt{Rr} is written with the cancellation already performed. No
$0\times\infty$ is ever evaluated, and the transformed integrand is finite
at $\tau=0$ with the value $L_{0}$ of Remark~\ref{rem:vertex}.

\emph{(b) Two hypergeometric representations on the $\pi/4$ ray.} There
$w^{4}=-\rho^{4}$, so $t=-4\rho^{4}/(1-\rho^{4})^{2}$ runs from $0$ to
$-\infty$. Setting
$K_{1}=(1-\rho^{4})^{-2a}F_{1}$, $K_{2}=(1-\rho^{4})^{-2a-1/2}F_{2}$ and
using $4a=\tfrac12-\beta$, every power of $(1-\rho^{4})$ cancels, and
$K_{1},K_{2}$ become analytic at $\rho=1$ --- expansions in integer powers
of $1-\rho^{4}$, because $b-a=\tfrac12$ excludes the logarithmic case. For
$\rho\le\tfrac45$ the series in $t$ is used directly; for $\rho>\tfrac45$
the $t\mapsto1/t$ connection formulae are used, with $z=1/t\to0$. The
constants \texttt{G1a},\dots,\texttt{G2b} are the standard connection
coefficients; $\Gamma(\tfrac12)=\sqrt\pi$ and $\Gamma(-\tfrac12)=-2\sqrt\pi$ 
are signatures of the geometry.

\emph{(c) The splice.} \texttt{DGb} carries the branch as an explicit
argument so that both representations can be evaluated at the \emph{same}
point, which is the only meaningful continuity test.

\subsection{Smoke tests}

\begin{verbatim}
Print["AR[1/2] = ", N[AR[1/2], 15]];
Print["AR[0]   = ", N[AR[0],   15]];       (* finite: the L0 limit *)
Print["DG[1/2] = ", N[DG[1/2], 15]];
Print["DG[1]   = ", N[DG[1],   15]];       (* finite: analytic endpoint *)
Print["splice residual at 4/5: ",
N[DGb[4/5, False] - DGb[4/5, True], 30]];
Print["L0 = ", N[AR[0][[2]], 20], "  vs  ",
N[4 Sqrt[2] pE 2^(-beta) F1at1 F2at1, 20]];
\end{verbatim}

\noindent
All four evaluations must return finite reals; if they do, the integrals
below cannot fail, both integrands being bounded on a closed interval.
\texttt{DG[1/2]} and \texttt{DG[1]} exhibit the pointwise identities
$J_{1}=-J_{2}$ and $J_{3}=0$ of Lemma~\ref{lem:cert}. The splice residual tests the
two hypergeometric representations of Part B.3(b) against each other at a
common point and should vanish to working precision; this is the only
ingredient of the pipeline with no other independent check.

An earlier form of this test, differencing
\texttt{DG[4/5-10\^{}-12]} and \texttt{DG[4/5+10\^{}-12]}, compares two
distinct points and therefore returns
$2\times10^{-12}\,DG'(4/5)\approx10^{-11}$; that number is a derivative
estimate, not a discontinuity.

\subsection{The two integrals}

\begin{verbatim}
ARc[k_Integer, tau_?NumericQ] := AR[tau][[k]];
DGc[k_Integer, rho_?NumericQ] := DG[rho][[k]];

Ivec = Table[NIntegrate[ARc[k, tau], {tau, 0, 1}, Evaluate@qopt], {k, 3}];
Jvec = Table[NIntegrate[DGc[k, rho], {rho, 0, 1}, Evaluate@qopt], {k, 3}];
\end{verbatim}

\noindent
This cell is the numerical heart of the paper: \texttt{Jvec} is
$X_{\Sigma}(e^{i\pi/4})=(J_{1},J_{2},J_{3})$, while \texttt{Ivec} is
$\bigl(I_{1},\,Y,\,I_{3}\bigr)$, so that
$X_{\Sigma}(1)=(I_{1},0,I_{3})$ and $Y$ is the quantity of Proposition~\ref{prop:twoint}. 
\textbf{Note the index trap:} \texttt{Ivec[[2]]} is $Y$, \emph{not}
the $I_{2}$ of Lemma~\ref{lem:cert}; the middle component of \texttt{AR} is the real
density $\Psi$, whereas the second component of $\Phi$ on the real ray is
$i\Psi$, purely imaginary, so the second coordinate of $X_{\Sigma}(1)$ is
zero without being computed. On the $\pi/4$ ray no such splitting occurs
and \texttt{Jvec} is literally $(J_{1},J_{2},J_{3})$.

The scalar wrappers \texttt{ARc}, \texttt{DGc} are not decorative. Passing
a vector-valued integrand guarded by \texttt{?NumericQ} directly to
\texttt{NIntegrate} fails: the symbol \texttt{tau} is still symbolic when
\texttt{Part} is applied, so \texttt{AR[tau][[k]]} either extracts a part
of the unevaluated expression tree (returning \texttt{tau} for $k=1$) or
raises \texttt{Part::partw} for $k=2,3$, after which
\texttt{NIntegrate} returns unevaluated and every downstream
\texttt{Print} displays the expression rather than a number. Wrapping the
part extraction inside a head whose argument is itself gated is the
portable idiom.

\subsection{Certificates and the constants of \S8}

\begin{verbatim}
Yv = Ivec[[2]];   Jv = Abs[Jvec[[2]]];

Print["straightness certificate = ",
N[{Jvec[[1]] - Ivec[[1]], Jvec[[2]], Jvec[[3]] - Ivec[[3]]}, 30]];
Print["I1 - 2|J| = ", N[Ivec[[1]] - 2 Jv, 30]];
Print["I3 -  |J| = ", N[Ivec[[3]] - Jv,   30]];
Print["J1 + J2   = ", N[Jvec[[1]] + Jvec[[2]], 30]];
Print["J3        = ", N[Jvec[[3]], 30]];

Print["Y          = ", N[Yv, 30]];
Print["|J|        = ", N[Jv, 30]];
Print["r          = ", N[Yv/Sqrt[3], 30]];
Print["ell        = ", N[2 Sqrt[2] Yv, 30]];
Print["L          = ", N[8 Sqrt[3] Jv, 30]];
Print["c_i = L/4  = ", N[2 Sqrt[3] Jv, 30]];
Print["Area       = ", N[4 Yv Jv, 30]];
Print["Area/ell^2 = ", N[Jv/(2 Yv), 30]];
Print["saving     = ", N[1 - 2 Jv/Yv, 30]];
Print["kappa      = ", N[kappa, 30]];
Print["Int K      = ", N[2 Pi - 4 ArcCos[-1/3], 30]];
Print["Brakke     = 0.217234155407559"];
\end{verbatim}

\noindent
The three entries of the straightness certificate must agree; they do so
in all $30$ printed digits at $\texttt{prec}=80$, and split in the
$29$th at $\texttt{prec}=60$. This is the strongest single check in the
paper, for the reason given in Lemma~\ref{lem:cert}, and it is obtained free from
integrals computed anyway. The four lines after it exhibit the
certificate in its equivalent form
$I_{1}=2|J|$, $I_{3}=|J|$, $J_{1}=-J_{2}$, $J_{3}=0$; the first two link
the two rays and are the substantive ones, the last two holding pointwise.

The absolute value in \texttt{Jv} implements the convention $\lambda=-i$
of \S2.4; without it, $L$, \texttt{Area} and
$\Area/\ell^{2}$ print with a negative sign, the orientation of
$\Sigma^{*}$ being reversed.

\subsection{Total curvature, and $\deg g=1$}

\begin{verbatim}
gpr[w_] := Block[{ww = SetPrecision[w, prec], t},
t = 4 ww^4/(1 + ww^4)^2;
Sqrt[2] kappa (1 - ww^4)^(beta - 1) (1 + ww^4)^(1/2 - beta)/
Hypergeometric2F1[a1, b1, c1, t]^2];

tk = -8 NIntegrate[
4 Abs[gpr[rho Exp[I phi]]]^2 rho/(1 + Abs[gOf[rho Exp[I phi]]]^2)^2,
{rho, 0, 1}, {phi, 0, Pi/4},
WorkingPrecision -> 30, PrecisionGoal -> 8, MaxRecursion -> 30,
Method -> {"GlobalAdaptive", "SymbolicProcessing" -> 0}];
Print["Int K by quadrature = ", N[tk, 10],
"     exact ", N[2 Pi - 4 ArcCos[-1/3], 10]];
\end{verbatim}

\noindent
This is the only two-dimensional integral in the verification, and the
only one limited to about eight digits; it uses the Wronskian closed form
for $g'$ of \S2.6, so no hypergeometric function is
differentiated. Its purpose is not to determine $\int K$, which
Proposition~\ref{prop:extocrv} gives exactly, but to confirm that the Gauss map covers
the spherical cuboctahedral square \emph{once}: a value near
$-1.3593$ rather than $-2.7187$ establishes $\deg g=1$. The integrand is
$O\bigl((1-\rho^{4})^{2\beta-2}\bigr)$ near the vertex, with exponent
$-0.784$, hence integrable; this check is insensitive to the vertex
truncation that limits boundary-based schemes, which is why it converged
readily even in the earliest versions of the computation.

\subsection{Coverage}

\begin{center}
	\begin{tabular}{llp{4.2cm}}
		\toprule
		Claim & Part & result\\
		\midrule
		$\kappa$ in $\Gamma$-values (\S2.5) & B.1, B.2 & tautology at the corner;
		independent to $18$ digits\\
		Gauss image is the cuboctahedral square (\S2.2, 3) & B.2 & midpoint to
		$\sim\!60$ digits\\
		orthogonality $N\cdot n=0$ (\S7) & B.2 & residual $O(10^{-30})$\\
		two straight sides of the triangle (\S2.2) & B.2 & exact to working
		precision\\
		$L_{0}$ (Remark 7.4) & B.4 & $20$ digits\\
		connection formulae (Part B.3b) & B.4 & splice residual at working
		precision\\
		two-integral reduction (Prop.\ 7.2) & B.5 & ---\\
		straightness, hence planarity of $\partial\Sigma^{*}$ (Lemma 7.3) & B.6 &
		$29$--$30$ digits\\
		$r,\ \ell,\ L,\ c_{i},\ \Area,\ \Area/\ell^{2}$ (\S8) & B.6 & $28$ digits\\
		$\deg g=1$, $\int K$ (Prop.\ 3.1) & B.7 & $8$ digits\\
		\bottomrule
	\end{tabular}
\end{center}

\noindent
\\
Two further points. An earlier computation by differential-equation
integration along the free boundary --- which gave $r$ to $12$ digits, the
four face normals with pairwise inner products $-\tfrac13$, and
$\Area(\Sigma)=\Area(\Sigma^{*})$ by isometry to $8\times10^{-8}$ ---
is superseded in accuracy by Parts B.2 and B.6, but not in logical
content: it tests the surface itself rather than the two rays, and is
retained in Part B.9 as an independent cross-check. Second, nothing in
this appendix addresses global minimality among volume-halving
competitors.

\subsection{Plot}
\begin{verbatim}
(* ================================================================= *)
(*  B.9  Figure 2: surface + tetrahedron + mid-edge square + cross  *)
(*  RUN AFTER B.1-B.6.   DO NOT ClearAll: this cell uses beta, a1,   *)
(*  b1, c1, kappa, pE and Yv from those cells.                       *)
(* ================================================================= *)
Off[NDSolve::precw];          (* precision bookkeeping, not an error *)
precP = 30;                   (* the picture needs ~6 digits *)

(* ---- integrand for the plot; left ungated, as NDSolve prefers ---- *)
PhiP[w_] := Block[{ww = SetPrecision[w, precP], w4, t, f1, f2, R, u1, u2, u3},
w4 = ww^4;  t = 4 w4/(1 + w4)^2;
f1 = Hypergeometric2F1[a1, b1, c1, t];
f2 = Hypergeometric2F1[a1 + 1/4, b1 + 1/4, 5/4, t];
R  = (1 - w4)^(-beta) (1 + w4)^beta;
u1 = 4 Sqrt[2] f1^2 R/(kappa Sqrt[1 + w4]);
u2 = 8 ww f1 f2 R/(1 + w4);
u3 = 8 Sqrt[2] kappa ww^2 f2^2 R/(1 + w4)^(3/2);
{(u1 - u3)/2, I (u1 + u3)/2, u2}];

(* ---- one ODE along the ray arg w = Pi/8 ---- *)
psi0 = SetPrecision[Pi/8, precP];
Rb   = SetPrecision[1 - 10^-12, precP];
solR = First@NDSolve[
{V'[r] == Exp[I psi0] PhiP[r Exp[I psi0]], V[0] == {0, 0, 0}},
V, {r, 0, Rb}, WorkingPrecision -> 24, PrecisionGoal -> 15, MaxSteps -> 10^6];

(* ---- one ODE per radius along the arc, in u with s = u^pE ---- *)
u0P   = psi0^(1 - beta);
uminP = SetPrecision[10^-8, precP];
umaxP = SetPrecision[Pi/4, precP]^(1 - beta);
arcU[rho_] := First@NDSolve[
{W'[u] == pE u^(pE - 1) I rho Exp[I u^pE] PhiP[rho Exp[I u^pE]],
W[u0P] == (V[rho] /. solR)},
W, {u, uminP, umaxP}, WorkingPrecision -> 24, PrecisionGoal -> 12,
MaxSteps -> 10^6];

(* ---- the grid: radii clustered toward the free boundary ---- *)
nr = 40;  ns = 32;
uvals = Table[u, {u, uminP, umaxP, (umaxP - uminP)/ns}];
rvals = Table[Rb (1 - (1 - k/nr)^2), {k, 1, nr}];
LaunchKernels[];  DistributeDefinitions["Global`"];
rows = ParallelTable[
Block[{sol = arcU[SetPrecision[rho, precP]]},
Table[N[-Im[W[u] /. sol]], {u, uvals}]], {rho, rvals}];
grid = Prepend[rows, ConstantArray[{0., 0., 0.}, Length[uvals]]];   (* w = 0 *)
Print["Dimensions[grid] = ", Dimensions[grid]];
Print["tip from grid    = ", N[grid[[-1, 1]], 12],
"     exact {0,-Y,0} = ", N[{0, -Yv, 0}, 12]];

(* ---- explicit quadrilateral mesh: no point-cloud reconstruction ---- *)
quadIdx[m_, n_] := Flatten[Table[
{(i - 1) n + j, (i - 1) n + j + 1, i n + j + 1, i n + j},
{i, m - 1}, {j, n - 1}], 1];
patch[pts_] := GraphicsComplex[Flatten[pts, 1],
Polygon[quadIdx[Length[pts], Length[First[pts]]]]];

(* ---- the D_2d orbit, the four faces, the vertices of T' ---- *)
nn0   = N[{1, 1, 1}/Sqrt[3], 30];
grp   = FixedPoint[Union[Flatten[Outer[Dot, #, #, 1], 1]] &,
{IdentityMatrix[3], DiagonalMatrix[{-1, 1, -1}],
{{0, 1, 0}, {1, 0, 0}, {0, 0, 1}}}];
nrms  = DeleteDuplicates[N[#, 25] & /@ (grp . nn0), Norm[#1 - #2] < 10^-15 &];
verts = N[Sqrt[3] Yv #] & /@ nrms;                    (* = 3 r n_i *)
sq    = N[{{Yv, 0, 0}, {0, Yv, 0}, {-Yv, 0, 0}, {0, -Yv, 0}}];
cross = N[{{{-Yv, 0, 0}, {Yv, 0, 0}}, {{0, -Yv, 0}, {0, Yv, 0}}}];

(* ---- Figure 2 ---- *)
fig = Graphics3D[{
(* tetrahedron T' *)
{Opacity[.10], EdgeForm[Directive[Thin, Gray]], FaceForm[LightBlue],
Polygon[Subsets[verts, {3}]]},
(* the surface: 8 copies of the fundamental patch *)
{FaceForm[Orange, Orange], Specularity[White, 40], EdgeForm[],
Table[patch[Map[m . # &, grid, {2}]], {m, grp}]},
(* flat mid-edge square, and the cross it shares with Sigma* *)
{Opacity[.22], FaceForm[LightGreen], Polygon[sq]},
{Directive[Red, Thickness[.005]], Line /@ cross},
{Black, PointSize[.011], Point[sq]}},
Boxed -> False, Lighting -> "Neutral", ImageSize -> 800,
ViewPoint -> {1.6, -2.2, 1.1}]
\end{verbatim}

\clearpage

\begin{figure}[H]
	\vspace*{3cm}   
	\centering
	\includegraphics[width=\textwidth,height=\textheight,keepaspectratio]{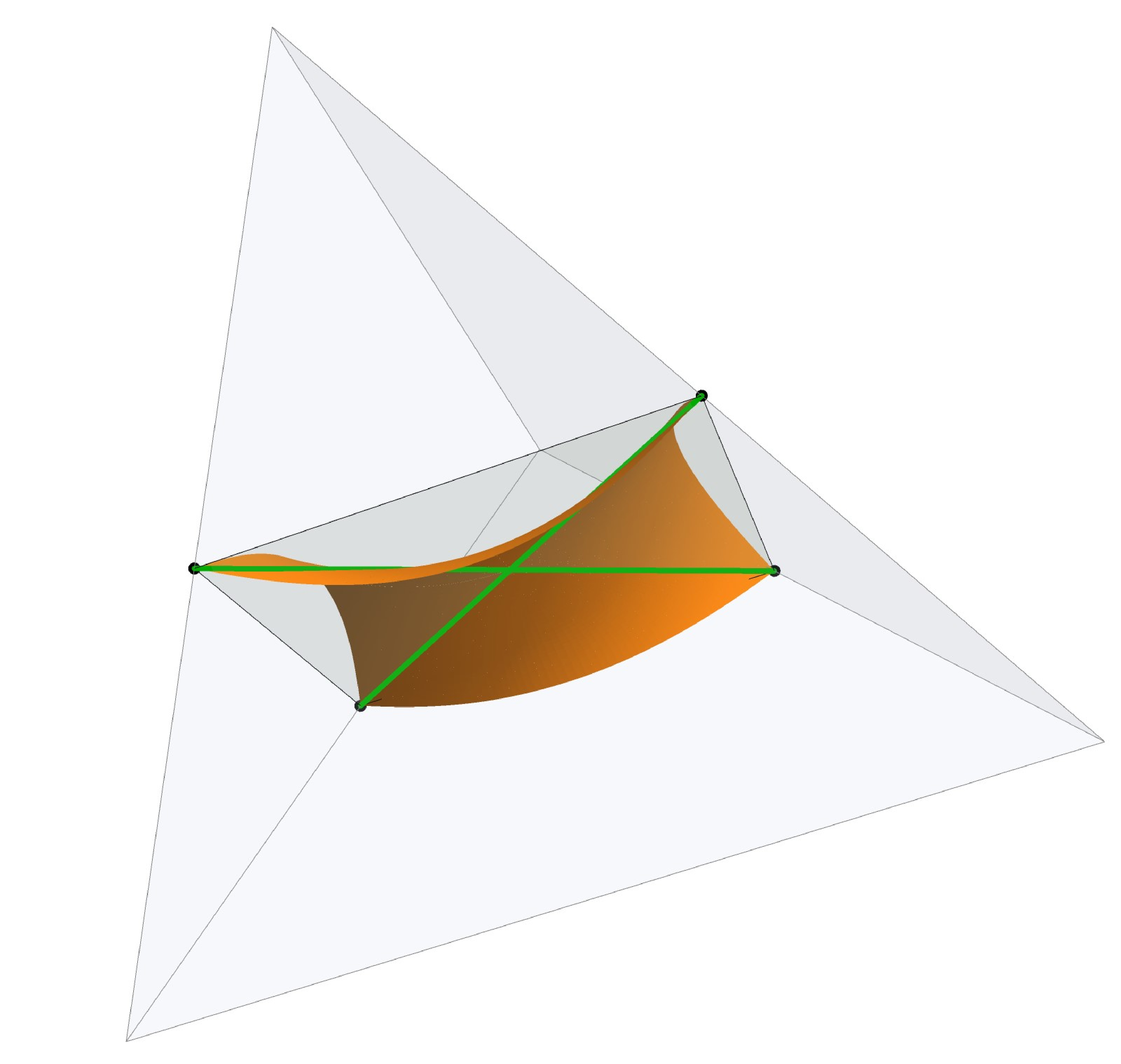}
	\caption[Minimal disk in a tetrahedron]{Another view of Smyth's free--boundary minimal disk 
		$\Sigma^{*}$ in $\tilde{T}$. \\ $\Sigma^{*}$ contains the two diagonals of the flat 
		mid-edge square, so the two surfaces share not only their corners but this 
		entire cross, and both bisect the volume; the four boundary arcs undulate
		alternately to either side of the square's plane.}
\end{figure}

\clearpage

\subsection*{Acknowledgements}
I thank Ken Brakke for helpful correspondence and Surface Evolver calculations. 
Anthropic Claude Opus 5.0 played a crucial role in turning 
my conjectures into theorems: to say that the AI model merely ``assisted'' 
would be an understatement.  It was used in an absolutely essential research 
capacity throughout every aspect of the paper.  No AI funding whatsoever for 
this study was received.  My interest in minimal surfaces began years ago 
(see \cite{Finch}).


\begin{thebibliography}{12}

\bibitem{Smyth1984}
B.~Smyth,
\emph{Stationary minimal surfaces with boundary on a simplex},
Invent.\ Math.\ \textbf{76} (1984), 411--420.

\bibitem{Choe2025}
J.~Choe,
\emph{The periodic Plateau problem and its application},
Ann.\ Global Anal.\ Geom.\ (2025), Art.~16; arXiv:2104.09087.

\bibitem{Schwarz}
H.~A.~Schwarz,
\emph{Gesammelte Mathematische Abhandlungen}, Springer, Berlin, 1890.

\bibitem{Nehari}
Z.~Nehari,
\emph{Conformal Mapping}, Dover, 1975 (Schwarz triangle functions,
Ch.~V).

\bibitem{Karcher}
H.~Karcher,
\emph{Embedded minimal surfaces derived from Scherk's examples},
Manuscripta Math.\ \textbf{62} (1988), 83--114.

\bibitem{FogdenHyde}
A.~Fogden and S.~T.~Hyde,
\emph{Parametrization of triply periodic minimal surfaces}, Acta Cryst.\
A\textbf{48} (1992); and Eur.\ Phys.\ J.\ B \textbf{7} (1999), 91--104.

\bibitem{ChenWeber}
H.~Chen and M.~Weber,
\emph{A new deformation family of Schwarz' $\DD$ surface},
Trans.\ Amer.\ Math.\ Soc.\ \textbf{374} (2021), 2785--2803;
arXiv:1804.01442.

\bibitem{GKD} P.\,J.\,F. Gandy, D. Cvijovi\'c, A.\,L. Mackay and
J. Klinowski, \emph{Exact computation of the triply periodic
	\textup{D} (``diamond'') minimal surface},
Chem.\ Phys.\ Lett.\ \textbf{314} (1999) 543--551.

\bibitem{AnselmoNelsonCrowdy}
T.~Anselmo, R.~Nelson, B.~Carneiro da Cunha and D.~G.~Crowdy,
\emph{Accessory parameters in conformal mapping: exploiting the
isomonodromic tau function for Painlev\'e VI},
Proc.\ R.\ Soc.\ A \textbf{474} (2018), 20180080.

\bibitem{Finch} S. R. Finch, \emph{Mathematical Constants II}, Cambridge Univ.\ 
Press, 2019; \S5.17, Gergonne-Schwarz surface; \S5.18, Partitioning problem.

\end{thebibliography}
\end{document}